\documentclass[a4paper,10pt, oneside]{article}
\usepackage{amsmath,amssymb,amsthm,mathrsfs,graphicx}
\usepackage[inline]{enumitem}
\usepackage[font=small,labelfont=md,textfont=it]{caption}
\usepackage{floatrow,float}
\usepackage[titletoc, title]{appendix}
\usepackage[colorlinks,linkcolor=blue,citecolor=blue]{hyperref}
\usepackage{etoolbox}
\usepackage{longtable}
\usepackage{diagbox}
\usepackage{booktabs,makecell,multirow}
\usepackage[capitalise, sort]{cleveref}
\usepackage{cases,color}
\crefname{equation}{}{}
\crefname{lem}{Lemma}{Lemmas}
\crefname{thm}{Theorem}{Theorems}
\crefname{discr}{Discretization}{Discretizations}
\usepackage{authblk}

\DeclareMathOperator{\D}{D}

\apptocmd{\sloppy}{\hbadness 10000\relax}{}{}

\newcommand{\dual}[1]{\langle {#1} \rangle}

\newcommand{\nm}[1]{\lVert {#1} \rVert}
\newcommand{\Nm}[1]{\left\lVert {#1} \right\rVert}
\newcommand{\snm}[1]{\lvert {#1} \rvert}

\newcommand{\ssnm}[1]
{
  \left\vert\kern-0.25ex
  \left\vert\kern-0.25ex
  \left\vert
  {#1}
  \right\vert\kern-0.25ex
  \right\vert\kern-0.25ex
  \right\vert
}

\makeatletter
\def\spher@harm#1{%
  \vbox{\hbox{%
    \offinterlineskip
    \valign{&\hb@xt@2\p@{\hss$##$\hss}\vskip.2ex\cr#1\crcr}%
  }\vskip-.36ex}%
}
\def\gshone{\spher@harm{.}}
\def\gshtwo{\spher@harm{.&.}}
\def\gshthree{\spher@harm{.&.&.}}
\let\gsh\spher@harm
\makeatother

\newtheorem{discr}{Discretization}

\newtheorem{lem}{Lemma}[section]
\newtheorem{rem}{Remark}[section]
\newtheorem{thm}{Theorem}[section]

\makeatletter\def\@captype{table}\makeatother

\begin{document}

\title{
  \Large \bf Analysis of the L1 scheme for fractional wave equations with nonsmooth data
  \thanks
{
	Binjie Li was supported in part by the National Natural Science
	Foundation of China (NSFC) Grant No.~11901410, Tao Wang was supported in part by the China Postdoctoral Science Foundation (CPSF) Grant No.~2019M66294, and Xiaoping Xie
	was supported in part by the National Natural Science Foundation of
	China (NSFC) Grant No.~11771312.}}
\author[1]{Binjie Li\thanks{libinjie@scu.edu.cn, libinjie@aliyun.com}}
\author[2]{Tao Wang\thanks{Corresponding author: wangtao5233@hotmail.com}}
\author[1]{Xiaoping Xie\thanks{xpxie@scu.edu.cn}}
\affil[1]{School of Mathematics, Sichuan University, Chengdu 610064, China}
\affil[2]{South China Research Center for Applied Mathematics and Interdisciplinary Studies, South China Normal University, Guangzhou 510631, China}

\maketitle

\begin{abstract}
	This paper analyzes the well-known L1 scheme for fractional wave equations with
	nonsmooth data. A new stability estimate is obtained, and the temporal accuracy $
	\mathcal O(\tau^{3-\alpha}) $ is derived for the nonsmooth data. In addition,  a modified L1 scheme is proposed,  stability and temporal accuracy $ \mathcal O(\tau^2) $ are derived
	for this scheme with nonsmooth data. The convergence of these
	schemes in inhomogeneous case are also established.  Finally,
	numerical experiments are performed to verify the theoretical results.
\end{abstract}

\section{Introduction}
Let $ 1 < \alpha < 2 $ and $ \Omega \subset \mathbb R^d $ ($d=1,2,3$) be a convex $ d
$-polytope. We consider the following fractional wave equation:
\begin{equation}
\label{eq:model}
\D_{0+}^{\alpha-1}(u'-u_1)(t) - \Delta u(t) = f(t),
\quad t > 0,
\end{equation}
subjected to the initial value condition $ u(0) = u_0 $, where $ u(t) \in
H_0^1(\Omega) $ for all $ t > 0 $, $ u_0 $, $ u_1 $ and $ f $ are given functions,
and $ \D_{0+}^{\alpha-1} $ is a Riemann-Liouville fractional differential operator of
order $ \alpha-1 $.

As a extension of integer order equation; the fractional diffusion and wave equations are widely used to model some processes with non-local effect, see \cite{nigmatullin1986the,bouchaud1990anomalous,carpinteri2011a,gerolymatou2006modelling}. We also refer readers to \cite{Kilbas2006Theory} for more background of fractional differential equations. By now there is an extensive literature on the numerical treatment of fractional
diffusion and wave equations.
Some of these researches give the convergence result under the condition that the solution is a $C^2$- or $C^3$- function in time. However, it is well known that the solution of a fractional diffusion (or wave) equation
generally has singularity in time despite how smooth the inital data is \cite{Jin2016}. In fact, the main challenge  is to design stable numerical scheme and to derive  convergence result, without regularity
restrictions on the solution, especially for the case with nonsmooth data.

Let us give a brief introduction of two kinds of numerical methods for solving fractional diffusion equations with nonsmooth data: the L1-type method \cite{Langlands2005The,Lin2007Finite,Gao2014A,Ren2017,Liao2018}, discontinuous Galerkin method \cite{Mustapha2015Time,Mustapha2012Superconvergence,Mustapha2011Piecewise,Adolfsson2003}. The L1-type method use L1 scheme to approximate the fractional derivative, these methods are very popular due to their ease of implementation. Jin et~al.~\cite{Jin2015} proved that
the L1 scheme is of temporal accuracy $ \mathcal O(\tau) $ for fractional
diffusion equations with smooth and nonsmooth initial data. Yan et~al.~\cite{Yan2018}
proposed a modified L1 scheme for fractional diffusion equations, which possesses
temporal accuracy $ \mathcal O(\tau^{2-\alpha}) $ for smooth and nonsmooth initial
data. The discontinuous Galerkin method use the finite element method to approximate the fractional derivative. McLean and Mustapha \cite{McLean2015} showed that the piecewise constant discontinuous Galerkin method is of temporal accuracy $ \mathcal O(\tau) $ for fractional
diffusion equations with  nonsmooth initial data. Li et~al. \cite{Li2018Analysis} investigate the regularity of fractional diffusion equations with nonsmooth data and they proved that discontinuous Galerkin method possesses optimal convergence rates in $L^2(0,T;L^2(\Omega))$ and $L^2(0,T;H^1(\Omega))$ norm, with respect to the regularity of the solution. For more related works, we refer reader to \cite{Yuste2005,li2009a,Chen2015spectral,Li2018Atime}.

Next, let us first briefly summarize some works on a variant of fractional wave equation:
\[
u'(t) - \Delta (\D_{0+}^{1-\alpha} u)(t) = u_1 + \D_{0+}^{1-\alpha} f(t),
\quad t > 0,
\]
which is obtained by applying $ \D_{0+}^{1-\alpha} $ to both sides of
\cref{eq:model}. For this equation,  McLean et~al.~\cite{McLean1993,McLean1996}
proposed two positive definite quadratures for the time fractional integral operator. Combing the
convolution quadratures in \cite{Lubich1988} and the backward
difference methods in time, Lubich et~al.~\cite{Lubich1996,Cuesta2006} proposed the first- and second-order time-stepping schemes and derived optimal error estimates with nonsmooth inital data.
Applying the famous discontinuous Galerkin method, Mustapha and
McLean~\cite{Mustapha2009Discontinuous} proposed a new class of algorithms. We note that the low-order algorithm in \cite{Mustapha2009Discontinuous} is identical to
the low-order algorithm proposed in \cite{McLean1996}.
For more related works, we
refer the reader to \cite{Cuesta2003,McLean2007}.

The study on fractional wave equation is limitied. Using the convolution quadratures in \cite{Lubich1988} and techniques in
\cite{Lubich1996}, Jin et~al.~\cite{Jin2016} developed first- and second-order time-stepping methods for
fractional wave equations and derived optimal error estimates with nonsmooth inital data.   In \cite{Luo2018Convergence}, the convergence in the $ H_0^1(\Omega) $-norm has been derived for  a low-order Petrov-Galerkin method with nonsmooth source term. We note that the low-order Petrov-Galerkin method in \cite{Luo2018Convergence} is identical to the L1 scheme.

As far as we know, the convergence in
the $ L^2(\Omega) $-norm of the L1 scheme for fractional wave equations with
nonsmooth data has not been established. In this paper, for a full discretization
using the L1 scheme in time and the standard $ P_1 $-element in space,  we derive a
new stability estimate and obtain the temporal accuracy $ \mathcal O(\tau^{3-\alpha})
$ in the $ L^2(\Omega) $-norm at positive times, with nonsmooth initial data. For
another full discretization using a modified L1 scheme in time and the $ P_1
$-element in space, we obtain the temporal accuracy $ \mathcal O(\tau^2) $ for
nonsmooth initial data. We also establish the convergence of the two
discretizations in inhomogeneous case (i.e., $ f \not\equiv 0 $). The derived error
estimates require that the temporal grid is uniform and that $ \tau^\alpha /
h_{\min}^2 $ is uniformly bounded, where $ h_{\min} $ is the minimum diameter of the
elements in the spatial triangulation. Our analysis implies that for nonzero initial
value $ u_0 $ large ratio $ \tau^\alpha/h_{\min}^2 $ will significantly worsen the
temporal accuracy of the L1 scheme, and this is confirmed by the numerical result. To
our knowledge, this interesting phenomenon is firstly reported in this paper.


The rest of this paper is organized as follows. \cref{sec:ode} establishes the
stability and convergence of the L1 scheme and a modified L1 scheme for a fractional
ordinary equation. \cref{sec:main} derives the stability and convergence of two full
discretizations for problem \cref{eq:model}, which use the L1 scheme and a modified
L1 scheme in time, respectively. \cref{sec:numer} performs several numerical
experiments to verify the theoretical results. Finally, \cref{sec:conclusion}
provides some concluding remarks.

\section{Preliminaries}
Let $ -\infty \leqslant a < b \leqslant \infty $ and assume that $ X $ is a separable
Hilbert space $ X $ with inner product $ (\cdot,\cdot)_X $. For any $ -\infty <
\gamma < 0 $, define
\begin{align*}
  (\D_{a+}^\gamma v)(t) &:= \frac1{\Gamma(-\gamma)}
  \int_a^t (t-s)^{-\gamma-1} v(s) \, \mathrm{d}s,
  \quad a < t < b, \\
  (\D_{b-}^\gamma v)(t) &:= \frac1{\Gamma(-\gamma)}
  \int_t^b (s-t)^{-\gamma-1} v(s) \, \mathrm{d}s,
  \quad a < t < b,
\end{align*}
for all $ v \in L^1(a,b;X) $, where $ \Gamma(\cdot) $ is the gamma function. For any
$ m \leqslant \gamma < m+1 $ with $ m \in \mathbb N $, define
\begin{align*}
  \D_{a+}^\gamma v &:= \D^{m+1} \D_{a+}^{\gamma-m-1} v, \\
  \D_{b-}^\gamma v &:= (-1)^{m+1} \D^{m+1} \D_{b-}^{\gamma-m-1} v,
\end{align*}
for all $ v \in L^1(a,b;X) $, where $ \D $ is the first order differential operator
in the distribution sense.

Then we introduce some properties of fractional calculus operators used in this
paper. Define
\begin{align*}
  {}_0H^1(a,b;X) := \big\{
    v \in L^2(a,b;X): v' \in L^2(a,b;X), \, \lim_{t \to {a+}} v(t) = 0
  \big\}, \\
  {}^0H^1(a,b;X) := \big\{
    v \in L^2(a,b;X): v' \in L^2(a,b;X), \, \lim_{t \to {b-}} v(t) = 0
  \big\}.
\end{align*}
Assume that $ 0 < \gamma < 1 $. Define
\begin{align*}
  {}_0H^\gamma(a,b;X) &:= [L^2(a,b;X), {}_0H^1(a,b;X)]_{\gamma,2}, \\
  {}^0H^\gamma(a,b;X) &:= [L^2(a,b;X), {}^0H^1(a,b;X)]_{\gamma,2},
\end{align*}
where $ [\cdot,\cdot]_{\gamma,2} $ means the interpolation space defined by the
famous $K$-method \cite{Tartar2007}. We use $ {}_0H^{-\gamma}(a,b;X) $ and $
{}^0H^{-\gamma}(a,b;X) $ to denote the dual spaces of $ {}^0H^\gamma(a,b;X) $ and $
{}_0H^\gamma(a,b;X) $, respectively. By \cite[Lemma 3.3]{Luo2018Convergence} we have
that, for any $ v \in L^2(a,b;X) $,
\[
  \nm{\D_{b-}^{-\gamma} v}_{{}^0H^\gamma(a,b;X)} \leqslant
  C \nm{v}_{L^2(a,b;X)},
\]
where $ C $ is a positive constant depending only on $ \gamma $. Therefore,
we can define $ \D_{a+}^{-\gamma}: {}_0H^{-\gamma}(a,b;X) \to L^2(a,b;X) $ by that
\[
  \int_a^b \big( \D_{a+}^{-\gamma}v(t), w(t) \big)_X \, \mathrm{d}t =
  \dual{v, \D_{b-}^{-\gamma} w}_{{}^0H^\gamma(a,b;X)}
\]
for all $ v \in {}_0H^{-\gamma}(a,b;X) $ and $ w \in L^2(a,b;X) $, where $
\dual{\cdot,\cdot}_{{}^0H^\gamma(a,b;X)} $ means the duality pairing between $
{}_0H^{-\gamma}(a,b;X) $ and $ {}^0H^\gamma(a,b;X) $. Moreover, it is evident that
\begin{equation}
  \label{eq:frac_regu}
  \nm{\D_{a+}^{-\gamma} v}_{L^2(a,b;X)} \leqslant
  C \nm{v}_{{}_0H^{-\gamma}(a,b;X)},
  \quad \forall v \in {}_0H^{-\gamma}(a,b;X),
\end{equation}
where $ C $ is a positive constant depending only on $ \gamma $.

\begin{lem}
  \label{lem:frac_calcu}
  Assume that $ v \in {}_0H^\gamma(a,b;X) $ and $ w \in {}^0H^\gamma(a,b;X) $, with $
  0 < \gamma < 1/2 $. Then
  \begin{small}
  \begin{align*}
    & C_1 \nm{\D_{a+}^\gamma v}_{L^2(a,b;X)}^2 \leqslant
    \int_a^b (\D_{a+}^\gamma v(t), \D_{b-}^\gamma v(t))_X \, \mathrm{d}t
    \leqslant C_2 \nm{\D_{a+}^\gamma v}_{L^2(a,b;X)}^2, \\
    & C_1 \nm{\D_{b-}^\gamma v}_{L^2(a,b;X)}^2 \leqslant
    \int_a^b (\D_{a+}^\gamma v(t), \D_{b-}^\gamma v(t))_X \, \mathrm{d}t
    \leqslant C_2 \nm{\D_{b-}^\gamma v}_{L^2(a,b;X)}^2, \\
    & \dual{\D_{a+}^{2\gamma} v, w}_{{}^0H^{\gamma}(a,b;X)} =
    \int_a^b (\D_{a+}^\gamma v(t), \D_{b-}^\gamma w(t))_X \, \mathrm{d}t =
    \dual{\D_{b-}^{2\gamma} w, v}_{{}_0H^{\gamma}(a,b;X)},
  \end{align*}
  \end{small}
  where $ C_1 $ and $ C_2 $ are two positive constants depending only on $ \gamma $.
\end{lem}
\begin{rem}
  For the proof of \cref{lem:frac_calcu}, we refer the reader to \cite{Ervin2006}.
  Assume that $ 0 < \gamma < 1/2 $. If $ v \in {}_0H^\gamma(a,b;X) $ and $ w \in
  {}^0H^\gamma(a,b;X) $ satisfy that $ \D_{a+}^{2\gamma}v \in L^p(a,b;X) $ and $ w
  \in L^{p/(p-1)}(a,b;X) $ for some $ 1 < p < \infty $, then
  \[
    \dual{\D_{a+}^{2\gamma}v, w}_{{}^0H^\gamma(a,b;X)} =
    \int_a^b (\D_{a+}^{2\gamma} v(t), w(t))_X \, \mathrm{d}t.
  \]
\end{rem}

Finally, we introduce some conventions as follows: $ H_0^1(\Omega) $ denotes the
usual Sobolev space, and $ H^{-1}(\Omega) $ is its dual space; the spaces $
{}_0H^\gamma(a,b; \mathbb R) $ and $ {}^0H^\gamma(a,b; \mathbb R) $ are abbreviated
to $ {}_0H^\gamma(a,b) $ and $ {}^0H^\gamma(a,b) $, respectively; $ C_\times $ means
a generic positive constant depending only on its subscript(s), and its value may
differ at each occurrence; for an interval $ \omega \subset \mathbb R $, the notation
$ \dual{p,q}_\omega $ denotes $ \int_\omega pq $ whenever $ pq \in L^1(\omega) $.

\section{Two discretizations of a fractional ordinary equation}
\label{sec:ode}
This section considers two discretizations of the following fractional ordinary
equation:
\begin{equation}
  \label{eq:ode-y}
  \D_{0+}^{\alpha-1}(y' - y_1)(t) + \lambda y(t) = f(t),
  \quad t > 0,
\end{equation}
subjected to the initial value condition $ y(0) = y_0 $, where $ y_0, y_1 \in \mathbb R $,
$ f \in L^1(0,\infty) \cap {}_0H^{(1-\alpha)/2}(0,\infty) $, and $ \lambda \geqslant 1 $
is a positive constant. Let $ \mu := \lambda \tau^\alpha/2 $ and define $ t_j := j \tau $
for each $ j \in \mathbb N $, where $ \tau $ is a positive constant. Applying the
L1-scheme proposed in \cite{Sun2006}, we obtain the first discretization of equation 
\cref{eq:ode-y}.
\begin{rem}
In order to obtain the error estimates of PDE \cref{eq:model}, $\lambda$ will be chosen as one of the eigen values of the discrete Laplace operator $-\Delta_h$ in the next section. 
\end{rem}

\begin{discr}
  \label{discr:ode-1}
  Let $ Y_0 = y_0 $; for each $ k \in \mathbb N $, the value of $ Y_{k+1} $ is
  determined by
  \begin{small}
  \begin{equation}
    \label{eq:Y}
    \begin{aligned}
      & (Y_1 - Y_0)(b_{k+1} - b_k) +
      \sum_{j=1}^k(Y_{j+1} - 2Y_j + Y_{j-1})
      (b_{k-j+1} - b_{k-j}) \\
      & \quad {} + \mu (Y_k + Y_{k+1}) =
      \tau^{\alpha-1} \int_{t_k}^{t_{k+1}} f(t) \, \mathrm{d}t +
      \tau y_1(b_{k+1} - b_k),
    \end{aligned}
  \end{equation}
  \end{small}
  where $ b_j := j^{2-\alpha}/\Gamma(3-\alpha) $, $ j \in \mathbb N $.
\end{discr}
\begin{rem}
  The above discretization is actually an variant of the temporal discretization in
  \cite{Sun2006}, but it is identical to a low-order Petrov-Galerkin method analyzed
  in \cite{Luo2018Convergence}.
\end{rem}
The second discretization is a simple modification of the first one.
\begin{discr}
  \label{discr:ode-2}
  Let $ \mathcal Y_0 = y_0 $; for each $ k \in \mathbb N $, the value of $ \mathcal
  Y_{k+1} $ is determined by
  \begin{small}
  \begin{equation}
    \label{eq:new-Y}
    \begin{aligned}
      & (\mathcal Y_1 - \mathcal Y_0)(\beta_{k+1} - \beta_k) +
      \sum_{j=1}^k(\mathcal Y_{j+1} - 2 \mathcal Y_j + \mathcal Y_{j-1})
      (\beta_{k-j+1} - \beta_{k-j}) \\
      & \quad {} + \mu (\mathcal Y_k + \mathcal Y_{k+1}) =
      \tau^{\alpha-1} \int_{t_k}^{t_{k+1}} f(t) \, \mathrm{d}t +
      \tau y_1(\beta_{k+1} - \beta_k),
    \end{aligned}
  \end{equation}
  \end{small}
  where $ \beta_1 = b_1 + 2\sin(\alpha\pi/2) \sum_{k=1}^\infty (2k\pi)^{\alpha-3} $
  and $ \beta_k := b_k $ for all $ k \in \mathbb N \setminus \{1\} $.
\end{discr}
\begin{rem}
In the numerical analysis of \cref{discr:ode-1} (cf. \cref{rem:motivation} and \cref{rem:motivation2}), we found that $(\widehat b(z) - z^{\alpha-3})(0) \neq 0
$
caused $ (3-\alpha) $-order accuracy of the first discretization,  where $\widehat b(z)$ is the discrete Laplace transform of $(b_k)_{k=0}^{\infty} $. This is the
motivation for the second discretization. Let $ \widehat\beta(z) $ be the discrete
Laplace transform of $ (\beta_k)_{k=0}^\infty $. The definition of the sequence $
(\beta_k)_{k=0}^\infty $ implies
\begin{align*}
\widehat\beta(z) &= \widehat b(z) + 2\sin(\alpha\pi/2)
\sum_{k=1}^\infty (2k\pi)^{\alpha-3} \\
&= \widehat b(z) - \big( \widehat b(z) - z^{\alpha-3} \big)(0)
\quad\text{(by \cref{eq:wt-b-2}).}
\end{align*}
Hence, $ \big( \widehat\beta(z) - z^{\alpha-3} \big)(0) = 0 $.
\end{rem}

In the rest of \cref{sec:ode}, we shall use the well-known Laplace transform technique to analyze \cref{discr:ode-1,discr:ode-2}. Firstly, we prove that the discrete Laplace transform of numerical solutions are well defined (i.e. they will not blow up in some places). Secondly, we give the integral representations of the exact and numerical solutions. Finally, we establish the error estimates by comparing the differences between the above two integrals.
\subsection{Stability of the two discretizations}
By an energy argument, it is easy to derive the following stability estimate of
\cref{discr:ode-1}.
\begin{lem}
  \label{lem:Y_k-stability}
  For each $ m \in \mathbb N_{>0} $,
  \begin{small}
  \begin{equation} \label{eq:Y-stability}
    \snm{Y_m} \leqslant C_\alpha \Big(
      \snm{y_0} + \lambda^{-1/2}
      \big(
        t_m^{1-\alpha/2} \snm{y_1} +
        \nm{f}_{{}_0H^{(1-\alpha)/2}(0,t_m)}
      \big)
    \Big).
  \end{equation}
  \end{small}
\end{lem}
\begin{proof}
  Multiplying both sides of \cref{eq:Y} by $ \tau^{1-\alpha}(Y_{k+1} - Y_k) $ and
  summing over $ k $ from $ 0 $ to $ m-1 $, we obtain
  \begin{align*}
    \dual{\D_{0+}^{\alpha-1} Y', Y'}_{(0,t_m)} +
    \lambda \dual{Y, Y'}_{(0,t_m)} =
    \dual{f, Y'}_{(0,t_m)} + y_1
    \dual{\D_{0+}^{(\alpha-1)} 1, Y'}_{(0,t_m)},
  \end{align*}
  where, for each $k \in \mathbb N$, $Y$ is linear on the interval $[t_k,t_{k+1}]$ and $
  Y(t_k) = Y_k $.
  A straightforward computation then gives
  \begin{small}
  \begin{align*}
    & \dual{\D_{0+}^{\alpha-1}Y', Y'}_{(0,t_m)} +
    \lambda \dual{Y,Y'}_{(0,t_m)} \\
    \leqslant{} &
    \snm{
      \dual{f, Y'}_{(0,t_m)} +
      y_1 \dual{\D_{0+}^{\alpha-1}1, Y'}_{(0,t_m)}
    } \\
    ={} &
    \snm{
      \dual{\D_{0+}^{(\alpha-1)/2} \D_{0+}^{(1-\alpha)/2} f, Y'}_{(0,t_m)}
    } + \snm{
      y_1 \dual{\D_{0+}^{(\alpha-1)/2} \D_{0+}^{(\alpha-1)/2} 1, Y'}_{(0,t_m)}
    } \\
    ={} &
    \snm{
      \dual{\D_{0+}^{(1-\alpha)/2} f, \D_{t_m-}^{(\alpha-1)/2}Y'}_{(0,t_m)}
    } +
    \snm{
      y_1 \dual{\D_{0+}^{(\alpha-1)/2} 1, \D_{t_m-}^{(\alpha-1)/2}Y'}_{(0,t_m)}
    } \\
    \leqslant{} &
    \Big(
      \nm{\D_{0+}^{(1-\alpha)/2} f}_{L^2(0,t_m)} +
      \snm{y_1} \nm{\D_{0+}^{(\alpha-1)/2} 1}_{L^2(0,t_m)}
    \Big) \nm{\D_{t_m-}^{(\alpha-1)/2} Y'}_{L^2(0,t_m)} \\
    \leqslant{} &
  C_{\alpha}  \Big(
      \nm{\D_{0+}^{(1-\alpha)/2} f}_{L^2(0,t_m)} +
      t_m^{1-\alpha/2} \snm{y_1}
    \Big) \nm{\D_{t_m-}^{(\alpha-1)/2} Y'}_{L^2(0,t_m)}.
  \end{align*}
  \end{small}
  Using integration by parts yields
  \[
    \dual{Y,Y'}_{(0,t_m)} = (Y_m^2 - Y_0^2)/2,
  \]
  and by \cref{lem:frac_calcu} we have
  \begin{small}
  \begin{align*}
    C_1 \nm{\D_{t_m-}^{(\alpha-1)/2} Y'}^2_{L^2(0,t_m)} \leqslant
    C_2 \nm{\D_{0+}^{(\alpha-1)/2} Y'}_{L^2(0,t_m)}^2 \leqslant
    \dual{\D_{0+}^{\alpha-1}Y', Y'}_{(0,t_m)},
  \end{align*}
  \end{small}
  where $ C_1 $ and $ C_2 $ are two positive constants depending only on $ \alpha $. By
  the above three estimates and the Young's inequality with $ \epsilon $, a simple
  calculation gives
  \begin{align*}
    \snm{Y_m} \leqslant C_\alpha
    \left(
      \snm{y_0} + \lambda^{-1/2}
      \Big(
        t_m^{1-\alpha/2} \snm{y_1} +
        \nm{\D_{0+}^{(1-\alpha)/2} f}_{L^2(0,t_m)}
      \Big)
    \right).
  \end{align*}
  Therefore, \cref{eq:frac_regu} implies \cref{eq:Y-stability} and thus concludes the
  proof.
\end{proof}

To derive the stability of \cref{discr:ode-2}, for $z \in \mathbb C_{+}:= \{w \in \mathbb
C: \operatorname{Re} w > 0 \}$, we introduce the discrete Laplace transform of $
(b_k)_{k=0}^\infty $ by that
\[
  \widehat b(z) := \sum_{k=0}^\infty b_k e^{-kz}.
\]
By the routine analytic continuation technique, $ \widehat b $ has a Hankel integral
representation (see \cite{Wood1992})
\begin{equation}
  \label{eq:512}
  \widehat b(z) = \frac1{2\pi i}
  \int_{-\infty}^{(0+)} \frac{w^{\alpha-3}}{e^{z-w}-1} \, \mathrm{d}w,
  \quad z \in \mathbb C \setminus (-\infty, 0],
\end{equation}
where $ \int_{-\infty}^{({0+})} $ means an integral on a piecewise smooth and
non-self-intersecting path enclosing the negative real axis and orienting
counterclockwise, $ 0 $ and $ \{z+2k\pi i \neq 0: k \in \mathbb Z\} $ lie on the different
sides of this path, and $ w^{\alpha-3} $ is evaluated in the sense that
\[
  w^{\alpha-3} = e^{(\alpha-3) \operatorname{Log}w}.
\]
Therefore, by Cauchy's integral theorem and Cauchy's integral formula, we have (see
\cite[(13.1)]{Wood1992})
\begin{equation}
  \label{eq:wt-b-2}
  \widehat b(z) = \sum_{k=-\infty}^\infty
  (z + 2k\pi i)^{\alpha-3}
\end{equation}
for all $ z \in \mathbb C \setminus (-\infty,0] $ satisfying $ -2\pi < \operatorname{Im} z
< 2\pi $. From \cref{eq:512} it follows that
\[
  \overline{\widehat b(z)} = \widehat b(\overline z) \text{ for all }
  z \in \mathbb C \setminus (-\infty,0].
\]
From \cref{eq:wt-b-2} it follows that
\[
  \widehat b(z) - z^{\alpha-3} \text{ is analytic on }
  \{w \in \mathbb C: \snm{\operatorname{Im} w} < 2\pi \}.
\]
\begin{rem}The $ \widehat b(z) $ also has another representation \cite{Wood1992}, 
$$
\widehat{b}(z)=\frac{\text{Li}_{\alpha-2}(e^{-z})}{\Gamma(3-\alpha)},
$$
where the polylogarithm is defined by  \[
\text{Li}_{p}(z)=\sum_{k=1}^{\infty} \frac{z^k}{k^p}, \quad \text{for} \ \snm{z}<1 \ \text{and} \ p\in \mathbb{C}.
\]

\end{rem}

\begin{lem}
  \label{lem:Y_k-stability-new}
  For each $ m \in \mathbb N_{>0} $,
  \begin{small}
  \begin{equation} \label{eq:Y-stability-new}
    \snm{\mathcal Y_m}  \leqslant C_\alpha \Big(
      \snm{y_0} + \lambda^{-1/2} \big(
        t_m^{1-\alpha/2} \snm{y_1} +
        \nm{f}_{{}_0H^{(1-\alpha)/2}(0,t_m)}
      \big)
    \Big).
  \end{equation}
  \end{small}
\end{lem}
\begin{proof}
  In virtue of the proof of \cref{lem:Y_k-stability}, it suffices to prove
  \begin{small}
  \begin{equation}
    \label{eq:ZZ}
    \sum_{k=0}^m Z_k \delta_k \leqslant
    C_\alpha \sum_{k=0}^m \mathcal Z_k \delta_k,
  \end{equation}
  \end{small}
  where
  \begin{small}
  \begin{align*}
    \delta_j &:=
    \begin{cases}
      \mathcal Y_{j+1} - \mathcal Y_j, & 0 \leqslant j < m, \\
      0, & m \leqslant j < \infty,
    \end{cases} \\
    Z_k &:= (b_{k+1} - b_k)\delta_0 +
    \sum_{j=1}^k(b_{k-j+1} - b_{k-j})(\delta_j - \delta_{j-1}), \\
    \mathcal Z_k &:= (\beta_{k+1} - \beta_k)\delta_0 +
    \sum_{j=1}^k(\beta_{k-j+1} - \beta_{k-j})(\delta_j - \delta_{j-1}).
  \end{align*}
  \end{small}

  To this end, we proceed as follows. For  $ z \in \mathbb C_{+} $, let $\widehat \beta(z)$, $ \widehat
  \delta(z) $, $ \widehat Z(z) $ and $ \widehat{\mathcal Z}(z) $ be the discrete
  Laplace transforms of $ (\beta_k)_{k=0}^\infty $, $ (\delta_k)_{k=0}^\infty $, $ (Z_k)_{k=0}^\infty $ and $
  (\mathcal Z_k)_{k=0}^\infty $, respectively. It is easy to verify that $\widehat \beta$, $ \widehat
  \delta $, $ \widehat Z $ and $ \widehat{\mathcal Z} $ are analytic on $ \mathbb
  C_{+} $. A straightforward computation gives that, for $ z \in \mathbb C_{+} $,
  \begin{align*}
    \widehat Z(z) = e^{-z}(e^z-1)^2 \widehat b(z) \widehat \delta(z), \\
    \widehat{\mathcal Z}(z) = e^{-z}(e^z-1)^2 \widehat \beta(z) \widehat \delta(z),
  \end{align*}
  and hence, by \cref{eq:wt-b-2} and the fact
  \[
    \widehat\beta(z) = \widehat b(z) +(\beta_1-b_1)e^{-z},
  \]
  we obtain
  \begin{small}
  \begin{align*}
    & \sup_{0 < x < 1} \int_{-\pi}^\pi
    \snm{\widehat Z(x+iy)}^2 \, \mathrm{d}y < \infty, \\
    & \sup_{0 < x < 1} \int_{-\pi}^\pi
    \snm{\widehat{\mathcal Z}(x+iy)}^2 \, \mathrm{d}y < \infty, \\
    & \lim_{x \to {0+}} \int_{-\pi}^\pi \snm{
      \widehat Z(x+iy) -
      e^{-iy}(e^{iy}-1)^2\widehat b(iy) \widehat\delta(iy)
    }^2 \, \mathrm{d}y = 0, \\
    & \lim_{x \to {0+}} \int_{-\pi}^\pi \snm{
      \widehat{\mathcal Z}(x+iy) -
      e^{-iy}(e^{iy}-1)^2\widehat \beta(iy) \widehat\delta(iy)
    }^2 \, \mathrm{d}y = 0.
  \end{align*}
  \end{small}
  Following the proof of the well-known Paley-Wiener Theorem \cite[Theorem
  1.8.3]{Arendt2011}), we easily conclude that
  \begin{small}
  \begin{align*}
    & \sum_{k=0}^\infty Z_k^2 < \infty, \quad
    \sum_{k=0}^\infty \mathcal Z_k^2 < \infty, \\
    & \sum_{k=0}^\infty Z_k e^{-iy} =
    e^{-iy}(e^{iy}-1)^2 \widehat b(iy) \widehat\delta(iy)
    \quad \text{ in } L^2(-\pi,\pi; \mathrm{d}y), \\
    & \sum_{k=0}^\infty \mathcal Z_k e^{-iy} =
    e^{-iy}(e^{iy}-1)^2 \widehat \beta(iy) \widehat\delta(iy)
    \quad \text{ in } L^2(-\pi,\pi; \mathrm{d}y).
  \end{align*}
  \end{small}
  Therefore, by the famous Parseval's theorem,
  \begin{small}
  \begin{align}
    \sum_{k=0}^m Z_k \delta_k =
    \sum_{k=0}^\infty Z_k \delta_k &= \frac1{2\pi}
    \int_{-\pi}^\pi e^{-iy}(e^{iy}-1)^2 \widehat b(iy)
    \snm{\widehat\delta(iy)}^2 \, \mathrm{d}y \notag \\
    &= \frac1\pi \int_0^\pi \operatorname{Re}\Big(
      e^{-iy}(e^{iy}-1)^2\widehat b(iy)
    \Big) \snm{\widehat\delta(iy)}^2 \, \mathrm{d}y.
    \label{eq:Z_k}
  \end{align}
  \end{small}
  Similarly,
  \begin{small}
  \begin{equation} \label{eq:cal_Z_k}
    \sum_{k=0}^m \mathcal Z_k \delta_k = \frac1\pi \int_0^\pi
    \operatorname{Re}\Big( e^{-iy}(e^{iy}-1)^2\widehat\beta(iy)\Big)
    \snm{\widehat\delta(iy)}^2 \, \mathrm{d}y.
  \end{equation}
  \end{small}

  In addition, a straightforward calculation gives, by \cref{eq:wt-b-2}, that
  \begin{small}
  \begin{align}
    & \operatorname{Re}\Big(
      e^{-iy}(e^{iy}-1)^2\widehat\beta(iy)
    \Big) \notag \\
    ={} &
    2(1\!-\!\cos y)\sin\big(\frac{\alpha\pi}2\big)
    \sum_{k=1}^\infty \Big(
      (2k\pi-y)^{\alpha-3} \!+\! (2k\pi+y-2\pi)^{\alpha-3}
      \!-\!  2\cos y (2k\pi)^{\alpha-3}
    \Big) \notag \\
    >{} & C_\alpha (1-\cos y) \sum_{k=1}^\infty
    \Big( (2k\pi-y)^{\alpha-3} + (2k\pi+y-2\pi)^{\alpha-3} \Big) \notag \\
    >{} & C_\alpha \operatorname{Re} \Big(
      e^{-iy}(e^{iy}-1)^2\widehat b(iy)
    \Big), \label{eq:b<beta}
  \end{align}
  \end{small}
  for all $ y \in [-\pi,\pi] \setminus \{0\} $. Finally, combining
  \cref{eq:Z_k,eq:cal_Z_k,eq:b<beta} yields \cref{eq:ZZ} and thus concludes the proof.
\end{proof}

\subsection{Convergence of \texorpdfstring{\cref{discr:ode-1}}{}}
\subsubsection{Integral representation of \texorpdfstring{$Y_k$}{}}
For any $ z \in \mathbb C_{+} $, let $ \widehat Y(z) $ be the discrete Laplace transform
of $ (Y_k)_{k=0}^\infty $.
In virtue of \cref{lem:Y_k-stability}, $ \widehat Y $ is analytic on $ \mathbb C_{+} $.
Multiplying both sides of \cref{eq:Y} by $ e^{-kz} $ and summing over $ k $ from $0$ to
$\infty$, we obtain
\begin{equation}
  \label{eq:lxy}
  \begin{aligned}
    (\psi(z) + \mu(e^z+1)) \widehat Y(z) & =
    ((e^z-1)^2\widehat b(z) + \mu e^z) y_0 +
    \tau(e^z-1)\widehat b(z) y_1 + \\
    & \quad {} + \tau^{\alpha-1} \sum_{k=0}^\infty
    \int_{t_k}^{t_{k+1}} f(t) \, \mathrm{d}t e^{-kz}, \quad \forall  z \in \mathbb C_{+},
  \end{aligned}
\end{equation}
where
\begin{align}
  \psi(z) &:= e^{-z}(e^z-1)^3 \widehat b(z)
  \label{eq:psi-1}.
\end{align}
By the properties of the function $ \widehat b $ in the previous subsection, $ \psi $ has
an analytic continuation as follows:
\begin{equation}
  \label{eq:psi}
  \psi(z) = e^{-z}(e^z-1)^3 \sum_{k=-\infty}^\infty
  (z + 2k\pi i)^{\alpha-3}
\end{equation}
for all $ z \in \mathbb C \setminus (-\infty,0] $ satisfying $ -2\pi <
\operatorname{Im} z < 2\pi $. Moreover,
\begin{equation}
  \label{eq:psi-conj}
  \overline{\psi(z)} = \psi(\overline z) \text{ for all }
  z \in \mathbb C \setminus (-\infty,0] \text{ with }
  -2\pi < \operatorname{Im} z < 2\pi,
\end{equation}
\begin{equation}
  \psi(z) - e^{-z}(e^z-1)^3 z^{\alpha-3} \text{ is analytic on }
  \{w \in \mathbb C: \snm{\operatorname{Im} w} < 2\pi \},
\end{equation}
and
\begin{equation}
  \label{eq:limt_0}
  \lim_{r \to {0+}} \frac{\psi(re^{i\theta})}{
    r^\alpha(\cos(\alpha\theta)+ i \sin(\alpha\theta))
  } = 1 \quad \text{uniformly for all $ -\pi < \theta < \pi $.}
\end{equation}

In the rest of \cref{sec:ode}, we assume that $\mu\leqslant \mu_0$, where $\mu_0$ is a given positive constant.
\begin{rem}
Let $\lambda$ be any eigen value of discrete Laplace operator $-\Delta_h$, then $\mu\leqslant \mu_0$ implies that $\tau^{\alpha}/h^2$ is bounded. The L1 scheme in \cref{discr:ode-1} reduces to the second order central difference scheme when $\alpha=2$, and this scheme require that $\tau/h$ is bounded (stability), which is consistent with the condition that $\mu \leqslant \mu_0$.
\end{rem}

\begin{lem}
  \label{lem:psi-esti}
  There exists $ \frac{\pi}{2} < \theta_{\alpha,\mu_0} \leqslant
  \frac{\alpha+2}{4\alpha}\pi $ depending only on $ \alpha $ and $ \mu_0 $ such that
  \begin{equation} \label{eq:psi-esti}
    \begin{aligned}
      \psi(z) + \mu(1+e^z) \neq 0 \quad \text{ for all }
      0 < \mu \leqslant \mu_0 \text{ and } \\
      z \in \{
        w \in \mathbb C :\,
        0 < \snm{\operatorname{Im} w} \leqslant \pi,
        \frac{\pi}{2} \leqslant \snm{\operatorname{Arg} w}
        \leqslant \theta_{\alpha,\mu_0}\}.
      \end{aligned}
	\end{equation}
\end{lem}
\begin{proof}
  By \cref{eq:limt_0}, there exists $ 0 < r_\alpha < \pi $, depending only on $
  \alpha $, such that $ \operatorname{Im} \left((1+e^z)^{-1} \psi(z)\right) > 0 $ and
  hence
  \begin{small}
  \begin{equation}
    \label{eq:426}
    \begin{aligned}
      \psi(z) \!+\! \mu (1+e^z)
      \neq 0 \quad \text{ for all }\, 0 < \mu \leqslant \mu_0 \text{ and } \\
      z \in \Big\{
        w \in \mathbb C:
        \frac{\pi}{2} \leqslant \operatorname{Arg} w \leqslant
        \frac{\alpha+2}{4\alpha}\pi, 0 < \operatorname{Im} w \leqslant r_\alpha
      \Big\}.
    \end{aligned}
  \end{equation}
  \end{small}
  From \cref{eq:psi-conj,eq:426}, it remains therefore to show that there exists
  $\frac{\pi}{2} < \theta_{\alpha,\mu_0} \leqslant \frac{\alpha+2}{4\alpha}\pi $ such
  that
  \begin{small}
  \begin{equation}
    \label{eq:427}
    \begin{aligned}
      \psi(z) \!+\! \mu (1+e^z)
      \neq 0 \quad \text{ for all }\, 0 < \mu \leqslant \mu_0
      \text{ and } \\
      z \in \Big\{
        w \in \mathbb C:
        \frac{\pi}{2} \leqslant \operatorname{Arg} w \leqslant
        \theta_{\alpha,\mu_0}, r_{\alpha}< \operatorname{Im} w \leqslant \pi
      \Big\}.
    \end{aligned}
  \end{equation}
  \end{small}
  To this end, we proceed as follows. For $ 0 < y \leqslant \pi $, by \cref{eq:psi} we have
  \begin{small}
  \begin{align}
    \psi(iy) & = e^{-iy}(e^{iy}-1)^3 \sum_{k=-\infty}^\infty
    (iy+2k\pi i)^{\alpha-3} \notag \\
    & = e^{-iy}(e^{iy}-1)^3 \Big(
      \sum_{k=-\infty}^{-1}  (-2k\pi - y)^{\alpha-3}
      (-i)^{\alpha-3} + \sum_{k=0}^\infty (2k\pi+y)^{\alpha-3} i^{\alpha-3}
    \Big) \notag \\
    &= e^{-iy}(e^{iy}-1)^3 \Big(
      \sum_{k=1}^\infty (2k\pi - y)^{\alpha-3}
      e^{i(3-\alpha)\pi/2} + \sum_{k=0}^\infty (2k\pi+y)^{\alpha-3}
      e^{-i(3-\alpha)\pi/2}
    \Big) \notag \\
    &=
    e^{-iy}(e^{iy}-1)^3 \Big(
      -\sum_{k=1}^\infty (2k\pi - y)^{\alpha-3}
      e^{i(1-\alpha)\pi/2} - \sum_{k=0}^\infty (2k\pi+y)^{\alpha-3}
      e^{-i(1-\alpha)\pi/2}
    \Big) \notag \\
    &= e^{-iy}(e^{iy}-1)^3(A(y) + iB(y)), \label{eq:psi-A-B}
  \end{align}
  \end{small}
  where
  \begin{small}
  \begin{align*}
    A(y) &:= - \cos((\alpha-1)\pi/2)
    \sum_{k=0}^\infty (2k\pi + 2\pi - y)^{\alpha-3} +
    (2k\pi + y)^{\alpha-3}, \\
    B(y) &:= \sin((\alpha-1)\pi/2) \sum_{k=0}^\infty
    (2k\pi+2\pi-y)^{\alpha-3} - (2k\pi+y)^{\alpha-3}.
  \end{align*}
  Moreover,
  \begin{small}
  \begin{equation}
    \label{eq:shit-2}
    \operatorname{Im}\left( (1+e^{iy})^{-1} \psi(iy) \right)=
    4A(y) \frac{(\cos y - 1)\sin y}{\snm{e^{iy} + e^{2iy}}^2} > 0,
    \quad \forall 0 < y < \pi.
  \end{equation}
  \end{small}
  \end{small}

  Inserting $y=\pi$ into \cref{eq:psi-A-B} yields $$ \psi(\pi i) = 8A(\pi) < 0 = \mu (1+e^{\pi i}),$$ so that by
  the continuity of $ \psi$,
  %
  there exists $ 0 < r_{\alpha,\mu_0}^1 \leqslant
  r_\alpha \tan((2-\alpha)/(4\alpha)\pi) $ and $ 0 < r_{\alpha,\mu_0}^2 < \pi $,
  depending only on $ \alpha $ and $ \mu_0 $, such that
  \begin{small}
  \begin{equation}
    \label{eq:732-2}
    \begin{aligned}
      \psi(z) + \mu(1+e^z) \neq 0 \quad
      \text{ for all }\, 0 < \mu \leqslant \mu_0 \text{ and } \\
      z \in \{
        w \in \mathbb C: -r_{\alpha,\mu_0}^1 \leqslant
        \operatorname{Re} w \leqslant 0, r_{\alpha,\mu_0}^2
        \leqslant \operatorname{Im} w \leqslant \pi
      \}.
    \end{aligned}
  \end{equation}
  \end{small}
  For the case of $ r_\alpha \leqslant \operatorname{Im} w \leqslant r_{\alpha,\mu_0}^2$, by $\eqref{eq:shit-2}$ and the continuity of $  \psi $,
  %
  it follows that there exists $ 0 < r_{\alpha,\mu_0}^3 \leqslant r_{\alpha,\mu_0}^1
  $, depending only on $ \alpha $ and $ \mu_0 $, such that $ \operatorname{Im}\left(
    (1+e^z)^{-1} \psi(z) \right)> 0 $ and hence
  \begin{small}
  \begin{equation}
    \label{eq:731-1}
    \begin{aligned}
      \psi(z) + \mu (1+e^z) \neq 0 \quad \text{ for all }
      0 < \mu \leqslant \mu_0 \text{ and } \\
      z \in \{
        w \in \mathbb C:
        -r_{\alpha,\mu_0}^3 \leqslant \operatorname{Re} w  \leqslant 0,\,
        r_\alpha \leqslant \operatorname{Im} w \leqslant r_{\alpha,\mu_0}^2
      \}.
    \end{aligned}
  \end{equation}
  \end{small}
  Finally, letting $ \theta_{\alpha,\mu_0} := \pi/2 + \arctan(r_{\alpha,\mu_0}^3/\pi)
  $ yields \cref{eq:427}, by \cref{eq:731-1,eq:732-2}. This completes the proof.
\end{proof}
\begin{rem} \label{rem:phi}
The $r^{1}_{\alpha,\mu_0}$ in the above proof will approximate 0, when $\mu_0 \rightarrow  \infty$. Hence, $\theta_{\alpha,\mu_0}  \rightarrow (\pi/2)+$ as $\mu_0 \rightarrow  \infty$.
\end{rem}

\begin{lem}
  \label{lem:neq0}
  For each $ z \in \mathbb C_{+} $ and $ \mu > 0 $,
  \begin{equation}
    \psi(z) + \mu(e^z+1) \neq 0.
  \end{equation}
\end{lem}
\begin{proof}
  Assume that $ z \in \mathbb C_{+} $ satisfies that
  \begin{equation}
    \label{eq:contradiction}
    \psi(z) + \mu(e^z + 1) = 0.
  \end{equation}
  It follows that
  \[
    \widehat b(z) = -\mu e^z (e^z+1) (e^z-1)^{-3},
  \]
  and hence
  \[
    (e^z-1)^2 \widehat b(z) + \mu e^z =
    -2\mu e^z(e^z-1)^{-1}.
  \]
  In the case that $ y_0 = 1 $, $ y_1 = 0 $ and $ f \equiv 0 $, from
  \cref{eq:lxy,eq:contradiction} we obtain
  \[
    (e^z-1)^2 \widehat b(z) + \mu e^z = 0.
  \]
  Since the above two equations are contradictory, this proves the lemma.
\end{proof}

\begin{rem}
The above two lemmas indicate that $\psi(z)+\mu(e^z+1)\ne0$ in some places. Hence, by  \cref{eq:lxy}, $\widehat Y(z)$ will not blow up in these places. Then it  is reasonable to give the integral representation of the numerical solution $Y$.
\end{rem}

For the sake of simplicity, in the rest of this subsection (i.e., Subsection 3.2) we use
the following conventions: $ \mu_0 $ is a positive constant and $ \mu \leqslant \mu_0 $; $
\theta_{\alpha,\mu_0} $ defined in \cref{lem:psi-esti} is abbreviated to $ \theta $.
Define
\begin{align*}
  \Upsilon &:= (\infty,0]e^{-i\theta} \cup [0,\infty)e^{i\theta}, \\
  \Upsilon_1 &:= \{z \in \Upsilon:\ \snm{\operatorname{Im} z} \leqslant \pi\},
\end{align*}
where $ \Upsilon $ is oriented so that $ \operatorname{Im} z $ increases along $ \Upsilon
$ and $\Upsilon_1$ inherit the orientation of $ \Upsilon $. In addition, if the integral
over $ \Upsilon / \Upsilon_1 $ is divergent, caused by the singularity of the underlying
integrand near the origin, then $ \Upsilon / \Upsilon_1 $ should be deformed so that the
origin lies at its left side; for example,
\[
  \Upsilon := (\infty,\epsilon]e^{-i\theta}
  \cup \{\epsilon e^{i\varphi}: -\theta \leqslant \varphi \leqslant \theta\}
  \cup [\epsilon,\infty) e^{i\theta},
\]
where $ \epsilon $ is an arbitrary positive constant.

\begin{lem}[\cite{Jin2016}]
  For any $ t > 0 $,
  \begin{equation}
    \label{eq:y}
    \begin{aligned}
      y(t) & = \frac1{2\pi i} \int_\Upsilon
      e^{(t/\tau)z} \frac{ y_0 z^{\alpha-1} + \tau y_1 z^{\alpha-2} }
      {z^\alpha+2\mu} \, \mathrm{d}z \\
      & \qquad {} + \int_0^t E(t-s)f(s) \, \mathrm{d}s,
    \end{aligned}
  \end{equation}
  where
  \begin{equation}
    \label{eq:E}
    E(t) := \frac{\tau^{\alpha-1}}{2\pi i} \int_\Upsilon
    e^{(t/\tau)z} (z^\alpha+2\mu)^{-1} \, \mathrm{d}z.
  \end{equation}
\end{lem}

\begin{lem}
  For each $ k \in \mathbb N_{>0}$,
  \begin{small}
  \begin{equation}
    \label{eq:Y_k}
    \begin{aligned}
      Y_k & = \frac1{2\pi i} \int_{\Upsilon_1}
      e^{kz} \frac{
        \big(
          (e^z-1)^2\widehat b(z) - \psi(z)/2 + \mu(e^z-1)/2
        \big) y_0 + \tau(e^z-1)\widehat b(z) y_1
      }{\psi(z) + \mu(e^z+1)} \, \mathrm{d}z \\
      & {} \qquad + \int_0^{t_k} \widetilde E(t_k-t) f(t) \, \mathrm{d}t,
    \end{aligned}
  \end{equation}
  \end{small}
  where
  \begin{equation}
    \label{eq:wtE}
    \widetilde E(t) := \tau^{\alpha-1} E_{\lceil t/\tau \rceil},
    \quad t > 0,
  \end{equation}
  with $ \lceil \cdot \rceil $ being the ceiling function and
  \begin{equation}
    \label{eq:E_j}
    E_j := \frac1{2\pi i} \int_{\Upsilon_1}
    e^{jz}(\psi(z) + \mu(e^z+1))^{-1} \, \mathrm{d}z,
    \quad for \ j \in \mathbb Z.
  \end{equation}
\end{lem}
\begin{proof}
  A straightforward computation yields, by \cref{eq:lxy,lem:neq0}, that
  \begin{small}
  \begin{equation}
    \widehat Y(z) = \frac{
      \big( (e^z-1)^2 \widehat b(z) \!+\! \mu e^z \big) y_0 \!+\!
      \tau (e^z-1) \widehat b(z) y_1 \!+\!  \tau^{\alpha-1}
      \sum_{j=0}^\infty \int_{t_j}^{t_{j+1}} \! f(t) \mathrm{d}t e^{-jz}
    }{ \psi(z) + \mu(e^z+1) },
  \end{equation}
  \end{small}
  for all $z \in \mathbb{C}_+$. Hence,
  \begin{align*}
    Y_k = \frac1{2\pi i} \int_{a-i\pi}^{a+i\pi}
    e^{kz} \widehat Y(z) \, \mathrm{d}z =
    \mathbb I_1 + \mathbb I_2 + \mathbb I_3, \quad for \  0 < a <\infty,
  \end{align*}
  where
  \begin{align*}
    \mathbb I_1 &:= \frac{y_0}{2\pi i} \int_{a-i\pi}^{a+i\pi}
    e^{kz} \frac{(e^z-1)^2\widehat b(z) + \mu e^z}{
      \psi(z) + \mu(e^z+1)
    } \, \mathrm{d}z, \\
    \mathbb I_2 &:= \frac{\tau y_1}{2\pi i} \int_{a-i\pi}^{a+i\pi}
    e^{kz} \frac{(e^z-1)\widehat b(z)}{\psi(z) + \mu(e^z+1)} \, \mathrm{d}z, \\
    \mathbb I_3 &:=\frac{\tau^{\alpha-1}}{2\pi i} \int_{a-i\pi}^{a+i\pi}
    e^{kz} \frac{
      \sum_{j=0}^\infty \int_{t_j}^{t_{j+1}} f(t) \, \mathrm{d}t e^{-jz}
    }{\psi(z) + \mu(e^z+1)} \, \mathrm{d}z.
  \end{align*}
  Here, by \cref{lem:neq0} and Cauchy's integral theorem we have
  \begin{align*}
    \mathbb I_1 & = \frac1{2\pi i} \int_{a-i\pi}^{a+i\pi}
    e^{kz} \frac{(e^z-1)^2\widehat b(z) + \mu e^z}
    {\psi(z) + \mu(e^z+1)} \, \mathrm{d}z \\
    & = \frac1{2\pi i} \int_{a-i\pi}^{a+i\pi}
    e^{kz} \Big(
      \frac{(e^z-1)^2\widehat b(z) + \mu e^z}
      {\psi(z) + \mu(e^z+1)} - \frac12
    \Big) \, \mathrm{d}z \\
    & = \frac1{2\pi i} \int_{a-i\pi}^{a+i\pi}
    e^{kz} \frac{
      (e^z-1)^2\widehat b(z) - \psi(z)/2 + \mu(e^z-1)/2
    }{\psi(z) + \mu(e^z+1)} \, \mathrm{d}z \\
    &=\frac{y_0}{2\pi i} \int_{\Upsilon_1} e^{kz} \frac{
      (e^z-1)^2\widehat b(z) - \psi(z)/2 + \mu(e^z-1)/2
    }{\psi(z) + \mu(e^z+1)} \, \mathrm{d}z,
  \end{align*}
  where the latter equality follows from \cref{lem:psi-esti} and the fact that
  \begin{align*}
    & e^{kz} \frac{
      (e^z-1)^2\widehat b(z) - \psi(z)/2 + \mu(e^z-1)/2
    }{\psi(z) + \mu(e^z+1)} \\
    ={} &
    e^{k(z+2\pi i)} \frac{
      (e^{z+2\pi i}-1)^2\widehat b(z+2\pi i) -
      \psi(z+2\pi i)/2 - \mu(e^{z+2\pi i} - 1)/2
    }{\psi(z+2\pi i) + \mu(e^{z+2\pi i}+1)}
  \end{align*}
  for  $\operatorname{Re} z \geqslant-\pi\cot(\theta)$ and $ \operatorname{Im} z = -\pi $.

  A similar argument gives
  \begin{equation}
    \begin{aligned}
      \mathbb I_2 & = \frac{y_1 \tau}{2\pi i} \int_{\Upsilon_1}
      e^{kz} \frac{
        (e^z-1)\widehat b(z)
      }{\psi(z) + \mu(e^z+1)} \, \mathrm{d}z.
    \end{aligned}
  \end{equation}
  We now turn to $ \mathbb I_3 $. Using Fubini's theorem and Cauchy's integral theorem, we have
  \begin{align*}
    \mathbb I_3 & = \frac{\tau^{\alpha-1}}{2\pi i} \int_{a-i\pi}^{a+i\pi}
    e^{kz} \frac{
      \sum_{j=0}^\infty \int_{t_j}^{t_{j+1}} f(t) \, \mathrm{d}t e^{-jz}
    }{\psi(z) + \mu(e^z+1)} \, \mathrm{d}z \\
    &=
    \sum_{j=0}^\infty \int_{t_j}^{t_{j+1}} f(t) \, \mathrm{d}t
    \frac{\tau^{\alpha-1}}{2\pi i} \int_{a-i\pi}^{a+i\pi}
    e^{(k-j)z}(\psi(z) + \mu(e^z+1))^{-1} \, \mathrm{d}z \\
    &=
    \sum_{j=0}^{k-1} \int_{t_j}^{t_{j+1}} f(t) \, \mathrm{d}t
    \frac{\tau^{\alpha-1}}{2\pi i} \int_{a-i\pi}^{a+i\pi}
    e^{(k-j)z}(\psi(z) + \mu(e^z+1))^{-1} \, \mathrm{d}z \\
    &=
    \sum_{j=0}^{k-1} \int_{t_j}^{t_{j+1}} f(t) \, \mathrm{d}t
    \tau^{\alpha-1} E_{k-j} \\
    &=
    \int_0^{t_k} f(t) \widetilde E(t_k-t) \, \mathrm{d}t.
  \end{align*}
  %
  Combining the estimates of $\mathbb I_1, \ \mathbb I_2$ and $\mathbb I_3$ proves \eqref{eq:Y_k} and hence the lemma.
\end{proof}

\subsubsection{Convergence for \texorpdfstring{$f \equiv 0$}{}}
\begin{lem}
  \label{lem:psi+mu}
  For each $ z \in \Upsilon_1 \setminus \{0\} $,
  \begin{equation}
    \label{eq:psi+mu}
    \snm{\psi(z) + \mu(1+e^z)} > C_{\alpha,\mu_0} (\mu + \snm{z}^\alpha).
  \end{equation}
\end{lem}
\begin{proof}
  By
  \cref{eq:psi} there exists a continuous function $ g $ on $ [0, \pi/\sin\theta]$
  such that
  \[
    (1+re^{i\theta})^{-1} \psi(re^{i\theta}) =
    r^\alpha e^{i\alpha\theta}/2 + r^{\alpha+1} g(r).
  \]
  It follows that
  \begin{align}
    & \snm{\mu + (1+re^{i\theta})^{-1} \psi(re^{i\theta})}^2 \notag \\
    ={} &
    \snm{\mu + r^\alpha e^{i\alpha\theta}/2 + r^{\alpha+1} g(r)}^2 \notag \\
    \geqslant{} &
    \snm{\mu + r^\alpha e^{i\alpha\theta}/2}^2/2 -
    r^{2(\alpha+1)} \snm{g(r)}^2 \notag \\
    ={} &
    (\mu+r^\alpha\cos(\alpha\theta)/2)^2/2 +
    r^{2\alpha}\sin(\alpha\theta)^2/8 -
    r^{2(\alpha+1)} \snm{g(r)}^2, \label{eq:shit-10}
  \end{align}
  and hence there exists $ 0 < r_{\alpha,\mu_0} < \pi/\sin\theta $, depending only on $
  \alpha $ and $ \mu_0 $, such that
  \[
    \snm{\mu + (1+re^{i\theta})^{-1}\psi(re^{i\theta})} >
    C_{\alpha,\mu_0} (\mu+r^\alpha) \quad
    \text{for all} \ 0 < r \leqslant r_{\alpha,\mu_0}.
  \]
  Therefore,
  \begin{align*}
    \inf_{
      \substack{
        0 < r \leqslant r_{\alpha,\mu_0}
      }
    } \frac{\snm{\mu + (1+re^{i\theta})^{-1}\psi(re^{i\theta})}}
    {\mu + r^\alpha} > C_{\alpha,\mu_0}.
  \end{align*}
  Using this estimate and
  \[
    \snm{1+re^{i\theta}} > C_{\alpha,\mu_0},
    \quad  \text{for all} \ 0 \leqslant r \leqslant \pi/\sin\theta,
  \]
  we have
  \begin{align*}
    \inf_{
      \substack{
        0 < r \leqslant r_{\alpha,\mu_0}
      }
    } \frac{\snm{\psi(re^{i\theta}) + \mu(1+r^{i\theta})}}
    {\mu + r^\alpha} > C_{\alpha,\mu_0}.
  \end{align*}
  In addition, applying the extreme value theorem yields, by \cref{eq:psi-esti}, that
  \[
    \inf_{
      \substack{
        r_{\alpha,\mu_0} \leqslant r \leqslant \pi/\sin\theta
      }
    } \frac{\snm{\psi(re^{i\theta}) + \mu(re^{i\theta} + 1)}}
    {\mu+r^\alpha} > C_{\alpha,\mu_0}.
  \]
  Together, the above two estimates show
  \[
    \inf_{
      \substack{
        0 < r \leqslant \pi/\sin\theta
      }
    } \frac{\snm{\psi(re^{i\theta}) + \mu(re^{i\theta} + 1)}}
    {\mu+r^\alpha} > C_{\alpha,\mu_0},
  \]
  which completes the proof.
\end{proof}

\begin{lem}
  \label{lem:shit-1}
  For each $ z \in \Upsilon_1 \setminus \{0\} $,
  \begin{equation} \label{eq:36}
    \snm{z + 2\mu z^{1-\alpha}} > C_{\alpha}
    (\snm{z} + \mu \snm{z}^{1-\alpha}).
  \end{equation}
\end{lem}
\begin{proof}
  A simple calculation yields
  \begin{align*}
    \snm{z+2\mu z^{1-\alpha}}&=\snm{z}\snm{1+2\mu z^{-\alpha}} \\ &=r\snm{1+2\mu r^{-\alpha}\cos(-\alpha \theta) +2i\mu r^{-\alpha}\sin(-\alpha \theta) } \\ &
    \geqslant C_{\alpha} \mu r^{1-\alpha}.
  \end{align*}
  Analogously, we have
  \begin{align*}
    \snm{z+2\mu z^{1-\alpha}}&=\snm{z}^{1-\alpha}\snm{z^{\alpha}+2\mu } \\ &=r^{1-\alpha}\snm{2\mu +r^{\alpha}\cos(\alpha \theta) +i r^{\alpha}\sin(\alpha \theta) } \\ &
    \geqslant C_{\alpha} r.
  \end{align*}
  Combining above two estimates proves \cref{eq:36} and hence the lemma.
\end{proof}
\begin{thm}
  \label{thm:y-Y-y0}
  For each $ k \in \mathbb N_{>0} $,
  \begin{equation}
    \label{eq:y-Y}
    \snm{y(t_k) - Y_k} \leqslant C_{\alpha,\mu_0}
    \tau^{3-\alpha} \big(
      t_k^{\alpha-3} \snm{y_0} + t_k^{\alpha-2} \snm{y_1}
    \big).
  \end{equation}
\end{thm}
\begin{proof}
  From \cref{eq:y,eq:Y_k}, it follows that
  \[
    y(t_k) - Y_k = \mathbb I_1 + \mathbb I_2 + \mathbb I_3,
  \]
  where
  \begin{align*}
    \mathbb I_1 &:= \frac1{2\pi i} \int_{\Upsilon\setminus\Upsilon_1}
    e^{kz} \frac{y_0 z^{\alpha-1} + \tau y_1 z^{\alpha-2} }{z^\alpha+2\mu} \, \mathrm{d}z, \\
    \mathbb I_2 &:= \frac{y_0}{2\pi i} \int_{\Upsilon_1}
    e^{kz}\Big(
      \frac{z^{\alpha-1}}{z^\alpha + 2\mu} - \frac{
        (e^z - 1)^2 \widehat b(z) - \psi(z)/2 + \mu(e^z-1)/2
      }{\psi(z) + \mu(e^z+1)}
    \Big) \, \mathrm{d}z, \\
    \mathbb{I}_3 &:=
    \frac{\tau y_1}{2\pi i}  \int_{ \Upsilon_1} e^{kz}
    \Big(
      \frac{z^{\alpha-2}}{z^\alpha + 2\mu} -
      \frac{\widehat b(z)(e^z-1)}{\psi(z)+\mu(e^z+1)}
    \Big) \ \mathrm{d}z.
  \end{align*}

  Let us first estimate $ \mathbb I_1 $. A simple calculation gives
  \begin{align*}
    \mathbb I_1 &= \frac1\pi \operatorname{Im}
    \int_{\pi/\sin\theta}^\infty e^{kre^{i\theta}}
    \frac{
      y_0 (re^{i\theta})^{\alpha-1} +
      \tau y_1 (re^{i\theta})^{\alpha-2}
    }{(re^{i\theta})^\alpha + 2\mu} e^{i\theta} \, \mathrm{d}r,
  \end{align*}
  and the fact $ \pi/2 < \theta < (\alpha+2)/(4\alpha)\pi $ implies
  \begin{align*}
    & \snm{
      \frac{
        y_0 (re^{i\theta})^{\alpha-1} +
        \tau y_1 (re^{i\theta})^{\alpha-2}
      }{(re^{i\theta})^\alpha + 2\mu} e^{i\theta}
    } \\
    \leqslant{} &
    \frac{
      \snm{y_0}r^{\alpha-1}  + \tau \snm{y_1} r^{\alpha-2}
    }{
      \snm{
        r^\alpha\cos(\alpha\theta) + 2\mu + ir^\alpha\sin(\alpha\theta)
      }
    } \\
    \leqslant{} & C_{\alpha,\mu_0}
    \big( \snm{y_0} r^{-1} + \tau \snm{y_1} r^{-2} \big).
  \end{align*}
  Hence,
  \begin{align}
    \snm{\mathbb I_1} & \leqslant
    C_{\alpha,\mu_0} \int_{\pi/\sin\theta}^\infty
    e^{kr\cos\theta} \big(
      \snm{y_0} r^{-1} + \tau \snm{y_1} r^{-2}
    \big) \, \mathrm{d}r \notag \\
    & \leqslant C_{\alpha,\mu_0} \big( \snm{y_0}k^{-1} + \tau \snm{y_1} k^{-1} \big)
    e^{k\pi\cot\theta}.
    \label{eq:esti-I1}
  \end{align}

  Then let us estimate $ \mathbb I_2 $. For $ z \in \Upsilon_1 \setminus \{0\} $,
  a straightforward calculation gives
  \begin{small}
  \begin{align*}
    & \snm{
      \psi(z) + \mu(1+e^z) -
      (z+2\mu z^{1-\alpha}) \Big(
        (e^z-1)^2\widehat b(z) - \psi(z)/2 +
        \mu(e^z-1)/2
      \Big)
    } \\
    < & C_\alpha \big(
      \snm{z}^{\alpha+2} + \mu\snm{z}^{3-\alpha} + \mu^2 \snm{z}^{2-\alpha}
    \big),
  \end{align*}
  \end{small}
  and so \cref{lem:psi+mu,lem:shit-1} imply
  \begin{align*}
    & \snm{
      \frac1{z + 2\mu z^{1-\alpha}} -
      \frac{
        (e^z-1)^2 \widehat b(z) - e^{-z}(e^z-1)^3 \widehat b(z)/2 +
        \mu(e^z-1)/2
      }{\psi(z) + \mu(1+e^z)}
    } \\
    <{} &
    C_{\alpha,\mu_0} \frac{
      \snm{z}^{\alpha+2} + \mu\snm{z}^{3-\alpha} + \mu^2\snm{z}^{2-\alpha}
    }{(\snm{z} + \mu \snm{z}^{1-\alpha})(\snm{z}^\alpha + \mu)}.
  \end{align*}
  It follows that
  \begin{align*}
    \snm{\mathbb I_2} \leqslant C_{\alpha,\mu_0}
    \snm{y_0} \int_0^{\pi/\sin\theta}
    e^{kr\cos\theta} \frac{
      r^{\alpha+2} + \mu r^{3-\alpha} + \mu^2 r^{2-\alpha}
    }{(r + \mu r^{1-\alpha})(r^\alpha + \mu)} \, \mathrm{d}r.
  \end{align*}
  If $0 < r < \mu^{1/\alpha}$ then
  \begin{align*}
    & \frac{
      r^{\alpha+2} + \mu r^{3-\alpha} + \mu^2 r^{2-\alpha}
    }{(r + \mu r^{1-\alpha})(r^\alpha + \mu)}\\ < {} &
    \mu^{-2} r^{\alpha-1}
    (r^{\alpha+2} + \mu r^{3-\alpha} + \mu^2 r^{2-\alpha}) \\
    ={} & \mu^{-2} r^{2\alpha+1} + \mu^{-1} r^2 + r <
    2r + r^{2-\alpha},
  \end{align*}
  and if $ \mu^{1/\alpha} < r$ then
  \begin{align*}
    & \frac{
      r^{\alpha+2} + \mu r^{3-\alpha} + \mu^2 r^{2-\alpha}
    }{(r + \mu r^{1-\alpha})(r^\alpha + \mu)} \\<{}&
    r^{-\alpha-1} (r^{\alpha+2} + \mu r^{3-\alpha} + \mu^2 r^{2-\alpha}) \\
    ={} &
    r + \mu r^{2-2\alpha} + \mu^2r^{1-2\alpha} <
    2r + r^{2-\alpha}.
  \end{align*}
  Therefore,
  \begin{equation}
    \label{eq:esti-I2}
    \snm{\mathbb I_2} \leqslant C_{\alpha,\mu_0}
    \snm{y_0} \int_0^{\pi/\sin\theta}
    e^{kr\cos\theta} r^{2-\alpha} \, \mathrm{d}r
    \leqslant C_{\alpha,\mu_0} \snm{y_0} k^{\alpha-3}.
  \end{equation}

  Finally, a similar argument as that to derive \cref{eq:esti-I2} yields
  \begin{equation}
    \label{eq:esti-I3}
    \snm{\mathbb I_3} \leqslant C_{\alpha,\mu_0} \tau k^{\alpha-2}\snm{y_1},
  \end{equation}
  and then combining \cref{eq:esti-I1,eq:esti-I2,eq:esti-I3} gives
  \begin{align*}
    \snm{y(t_k) - Y_k} & \leqslant C_{\alpha,\mu_0}
    \big( k^{\alpha-3} \snm{y_0} + \tau k^{\alpha-2} \snm{y_1} \big) \\
    & = C_{\alpha,\mu_0} \tau^{3-\alpha} \big(
      t_k^{\alpha-3} \snm{y_0} +
      t_k^{\alpha-2} \snm{y_1}
    \big),
  \end{align*}
  which proves \cref{eq:y-Y} and hence this theorem.
\end{proof}
\begin{rem} \label{rem:motivation}
In the above proof,   
	\begin{align*}
	& \snm{
		\psi(z) + \mu(1+e^z) -
		(z+2\mu z^{1-\alpha}) \Big(
		(e^z-1)^2\widehat b(z) - \psi(z)/2 +
		\mu(e^z-1)/2
		\Big)
	} \\
	< & C_\alpha \big(
	\snm{z}^{\alpha+2} + \mu\snm{z}^{3-\alpha} + \mu^2 \snm{z}^{2-\alpha}
	\big),
	\end{align*}
and the term $ \mu\snm{z}^{3-\alpha} $ leads to $ (3-\alpha) $-order accuracy. If we choose a $\beta$ 
such that \text{(i.e. ($\widehat \beta(z) -z^{\alpha-3} )=0$)}	\begin{align*}
& \snm{
	\Psi(z) + \mu(1+e^z) -
	(z+2\mu z^{1-\alpha}) \Big(
	(e^z-1)^2\widehat \beta(z) - \Psi(z)/2 +
	\mu(e^z-1)/2
	\Big)
} \\
< & C_\alpha \big(
\snm{z}^{\alpha+2} + \mu\snm{z}^{2} + \mu^2 \snm{z}^{2-\alpha}
\big),
\end{align*} then we can obtain $2 $-order accuracy, where $
\Psi(z) = e^{-z}(e^z-1)^3 \widehat \beta(z)$. This is the motivation of the second discretization.
\end{rem}

\subsubsection{Convergence for \texorpdfstring{$ y_0 = y_1 = 0 $}{}}
Define
\[
  \mathcal E(t) := \int_0^t (E-\widetilde E)(s) \, \mathrm{d}s,
  \quad t > 0,
\]
where $ E $ and $ \widetilde E $ are defined by \cref{eq:E} and \cref{eq:wtE},
respectively.
\begin{lem}
  \label{lem:cal_E_t_k}
  For any $ t_k < t \leqslant t_{k+1} $ with $ k \in \mathbb N $,
  \begin{equation}
    \label{eq:E-wtE*1}
    \snm{\mathcal E(t)} < C_{\alpha,\mu_0}
    \varepsilon(\alpha,\tau,k) \tau^{3-\alpha},
  \end{equation}
  where
  \begin{equation}
    \label{eq:varepsilon}
    \varepsilon(\alpha,\tau,k) :=
    \begin{cases}
      t_{k+1}^{2\alpha-3} & \text{ if } 1 < \alpha < 3/2, \\
      1+\snm{\ln\tau} & \text{ if } \alpha = 3/2, \\
      1 & \text{ if } 3/2 < \alpha < 2.
    \end{cases}
  \end{equation}
\end{lem}
\begin{proof}
  Since the proof of the case $ k = 0$ is simpler, we only prove the case $ k
  \geqslant 1$.
  By \cref{lem:psi-esti,lem:neq0} and the fact that
  \[
    (1-e^{-z})(\psi(z) + \mu(e^z+1)) =
    (1-e^{-(z+2\pi i)})(\psi(z+2\pi i) + \mu(e^{z+2\pi i}+1))
  \]
  for all $ z = x - i\pi $ with $ x \geqslant \pi \cot\theta $, applying Cauchy
  integral theorem yields that
  \[
    \int_{\Upsilon_1} \frac1{ (1-e^{-z})(\psi(z) + \mu(e^z+1)) }
    \, \mathrm{d}z = 0,
  \]
  and using Cauchy integral theorem again gives
  \[
    \int_\Upsilon \frac1{z(z^\alpha+2\mu)} \, \mathrm{d}z = 0.
  \]
  Therefore, from \cref{eq:E,eq:wtE} we have
  \begin{small}
  \begin{align*}
    & \mathcal E(t) =
    \int_0^{t_k} E(s) \, \mathrm{d}s -
    \sum_{j=1}^k \tau^\alpha E_j +
    \int_{t_k}^t E(s) \, \mathrm{d}s - (t-t_k) E_{k+1} \\
    ={} &
    \frac{\tau^\alpha}{2\pi i} \int_\Upsilon
    \frac{e^{kz}}{z(z^\alpha+2\mu)} \, \mathrm{d}z
    - \frac{\tau^\alpha}{2\pi i} \int_{\Upsilon_1}
    \frac{e^{kz} }{(1-e^{-z})(\psi(z) + \mu(e^z+1))} \, \mathrm{d}z \\
    & {} + \frac{\tau^\alpha}{2\pi i} \int_\Upsilon
    \frac{e^{(t/\tau)z}-e^{kz}}{z(z^\alpha+2\mu)} \, \mathrm{d}z -
    \frac{\tau^\alpha}{2\pi i} \int_{\Upsilon_1}
    \frac{(t/\tau-k)e^{(k+1)z}}{\psi(z) + \mu(e^z+1)} \, \mathrm{d}z.
  \end{align*}
  \end{small}
  Inserting $t=t_k$ into above equation yields
  \begin{equation}
    \label{eq:I-123}
    \mathcal E(t_k) = \mathbb I_1 + \mathbb I_2 + \mathbb I_3,
  \end{equation}
  where
  \begin{small}
  \begin{align*}
    \mathbb I_1 &:=
    \frac{\tau^\alpha}{2\pi i} \int_{\Upsilon \setminus \Upsilon_1}
    \frac{e^{t/\tau z}}{z(z^\alpha+2\mu)} \, \mathrm{d}z, \\
    \mathbb I_2 &:= \frac{\tau^\alpha}{2\pi i}
    \int_{\Upsilon_1} \frac{e^{kz}}{z(z^\alpha+2\mu)} -
    \frac{e^{kz} (1-e^{-z})^{-1}}{\psi(z) + \mu(e^z+1)} \, \mathrm{d}z, \\
    \mathbb I_3 &:=
    \frac{\tau^\alpha}{2\pi i} \int_{\Upsilon_1} e^{kz}
    \left(
      \frac{e^{(t/\tau-k)z} - 1}{z(z^\alpha+2\mu)} -
      \frac{(t/\tau-k) e^z}{\psi(z) + \mu(e^z+1)}
    \right) \, \mathrm{d}z.
  \end{align*}
  \end{small}

  It is clear that
  \begin{equation}
    \label{eq:I-1}
    \snm{\mathbb I_1} < C_{\alpha,\mu_0} \tau^\alpha
    \int_{\pi/\sin\theta}^\infty e^{t/\tau r\cos\theta}
    r^{-1-\alpha} \, \mathrm{d}r < C_{\alpha,\mu_0}
    \tau^{\alpha+1} t^{-1} e^{k\pi\cot\theta}.
  \end{equation}
  Let us proceed to estimate $ \mathbb I_2 $. For $ z \in \Upsilon_1 \setminus
  \{0\} $,  a simple calculation yields
  \begin{align*}
    \snm{
      \psi(z) + \mu(e^z+1) - z(z^\alpha+2\mu)(1-e^{-z})^{-1}
    } < C_{\alpha,\mu_0} (\mu \snm{z}^2 + \snm{z}^3),
  \end{align*}
  and \cref{lem:psi+mu,lem:shit-1} imply
  \begin{equation}
    \label{eq:112}
    \snm{
      z(z^\alpha+2\mu)(\psi(z) + \mu(e^z+1))
    } > C_{\alpha,\mu_0} \snm{z}(\snm{z}^{2\alpha} + \mu^2).
  \end{equation}
  Hence, if $ \mu^{1/\alpha} < \pi/\sin\theta $ then
  \begin{small}
  \begin{align*}
    \snm{\mathbb I_2} & < C_{\alpha,\mu_0} \tau^\alpha
    \bigg(
      \int_0^{\mu^{1/\alpha}} e^{kr\cos\theta}
      ( \mu^{-1}r + \mu^{-2} r^2 ) \, \mathrm{d}r + {} \\
      & \qquad\qquad\qquad\qquad \int_{\mu^{1/\alpha}}^{\pi/\sin\theta} e^{kr\cos\theta}
      (\mu r^{1-2\alpha} + r^{2-2\alpha}) \, \mathrm{d}r
    \bigg) \\
    & < C_{\alpha,\mu_0} \tau^\alpha
    \begin{cases}
      \int_0^{\pi/\sin\theta} e^{kr\cos\theta}
      r^{2-2\alpha} \, \mathrm{d}r
      & \text{ if } 1 < \alpha < 3/2, \\
      \int_0^{\mu^{1/\alpha}} \mu^{-1} r + \mu^{-2}r^2 \, \mathrm{d}r +
      \int_{\mu^{1/\alpha}}^{\pi/\sin\theta} r^{1-\alpha} + r^{2-2\alpha} \, \mathrm{d}r
      & \text{ if } 3/2 \leqslant \alpha <2,
    \end{cases} \\
    & < C_{\alpha,\mu_0} \tau^\alpha
    \begin{cases}
      k^{2\alpha-3} & \text{ if } 1 < \alpha < 3/2, \\
      1+\snm{\ln\tau} & \text{ if } \alpha=3/2, \\
      \tau^{3-2\alpha} & \text{ if } 3/2 < \alpha < 2,
    \end{cases}
  \end{align*}
  \end{small}
  and if $ \mu^{1/\alpha} \geqslant \pi/\sin\theta $ then
  \begin{small}
  \begin{align*}
    \snm{\mathbb I_2} & < C_{\alpha,\mu_0} \tau^\alpha
    \int_0^{\pi/\sin\theta} e^{kr\cos\theta}
    (\mu^{-1}r + \mu^{-2} r^2 ) \, \mathrm{d}r \\
    & < C_{\alpha,\mu_0} \tau^\alpha
    \int_0^{\pi/\sin\theta} e^{kr\cos\theta} r \, \mathrm{d}r <
    C_{\alpha,\mu_0} \tau^\alpha k^{-2}.
  \end{align*}
  \end{small}
  Consequently,
  \begin{equation}
    \label{eq:I-2}
    \snm{\mathbb I_2} <
    C_{\alpha,\mu_0} \varepsilon(\alpha,\tau,k) \tau^{3-\alpha}.
  \end{equation}

  Now, let us estimate $ \mathbb I_3 $. For $ z \in \Upsilon_1 \setminus \{0\} $,
  a routine calculation yields
  \begin{align*}
    & \snm{
      (e^{(t/\tau-k)z} - 1)(\psi(z) + \mu(e^z+1)) -
      z(z^\alpha+2\mu)(t/\tau - k) e^z
    } \\
    <{} & C_{\alpha,\mu_0} (t/\tau-k)
    \big( \snm{z}^{\alpha+2} + \mu\snm{z}^2 \big),
  \end{align*}
  so that by \cref{eq:112} we obtain
  \begin{small}
  \begin{align}
    \snm{\mathbb I_3} & < C_{\alpha,\mu_0} \tau^\alpha (t/\tau-k)
    \bigg(
      \int_0^{\min\{\mu^{1/\alpha},\pi/\sin\theta\}} e^{kr\cos\theta}
      (\mu^{-2}r^{\alpha+1} + \mu^{-1} r) \, \mathrm{d}r + {} \notag \\
      & \qquad\qquad \int_{\min\{\mu^{1/\alpha},\pi/\sin\theta\}}^{\pi/\sin\theta}
      e^{kr\cos\theta} (r^{1-\alpha} + \mu r^{1-2\alpha}) \, \mathrm{d}r
    \bigg) \notag \\
    & < C_{\alpha,\mu_0} \tau^\alpha(t/\tau-k)
    \int_0^{\pi/\sin\theta} e^{kr\cos\theta} r^{1-\alpha} \, \mathrm{d}r \notag \\
    & < C_{\alpha,\mu_0} \tau^\alpha (t/\tau-k) k^{\alpha-2}.
    \label{eq:I-3}
  \end{align}
  \end{small}

  Finally, combining \cref{eq:I-123,eq:I-1,eq:I-2,eq:I-3} proves \cref{eq:E-wtE*1} and
  thus concludes the proof.
\end{proof}

\begin{rem} \label{rem:motivation2}
In the above proof,  \begin{align*}
\snm{
	\psi(z) + \mu(e^z+1) - z(z^\alpha+2\mu)(1-e^{-z})^{-1}
} < C_{\alpha,\mu_0} (\mu \snm{z}^2 + \snm{z}^3),
\end{align*} and the term $ \snm{z}^3$ leads to (3-$\alpha$)-order accuracy. If we choose a $\beta$ such that \text{(i.e. ($\widehat \beta(z) -z^{\alpha-3} )=0$)}  \begin{align*}
\snm{
	\Psi(z) + \mu(e^z+1) - z(z^\alpha+2\mu)(1-e^{-z})^{-1}
} < C_{\alpha,\mu_0} (\mu \snm{z}^2 + \snm{z}^{2+\alpha}),
\end{align*} 
then we can obtain 2-order accuracy, where $
\Psi(z) = e^{-z}(e^z-1)^3 \widehat \beta(z)$.
\end{rem}


\begin{thm}
  \label{thm:y-Y-f}
  For each $ k \in \mathbb N_{>0} $, if $ f' \in L^1(0,t_k) $ then
  \begin{equation}
    \label{eq:y-Y-f}
    \begin{aligned}
      \snm{y(t_k) - Y_k} & \leqslant
      C_{\alpha,\mu_0} \tau^{3-\alpha}
      \varepsilon(\alpha,\tau,k) \snm{f(0)} + {} \\
      & \, C_{\alpha,\mu_0} \tau^{3-\alpha}
      \begin{cases}
        \int_0^{t_k} (t_{k+1}-t)^{2\alpha-3}
        \snm{f'(t)} \, \mathrm{d}t & \text{ if }
        1 < \alpha < 3/2, \\
        (1+\snm{\ln\tau}) \nm{f'}_{L^1(0,t_k)} & \text{ if }
        \alpha = 3/2, \\
        \nm{f'}_{L^1(0,t_k)} & \text{ if }
        3/2 < \alpha < 2,
      \end{cases}
    \end{aligned}
  \end{equation}
  where $ \varepsilon(\alpha,\tau,k) $ is defined by \cref{eq:varepsilon}.
\end{thm}
\begin{proof}
  By \cref{eq:y,eq:Y_k}, a straightforward computation yields that
  \begin{align*}
    y(t_k) - Y_k &= \int_0^{t_k} (E-\widetilde E)(t_k-t) f(t) \, \mathrm{d}s \\
    &= \int_0^{t_k} (E-\widetilde E)(t_k-t)
    \left( f(0) + \int_0^t f'(s) \, \mathrm{d}s \right) \, \mathrm{d}t \notag \\
    &= f(0) \mathcal E(t_k) + \int_0^{t_k} \mathcal E(t_k-t)f'(t) \, \mathrm{d}t
  \end{align*}
  for each $ k \in \mathbb N_{>0} $. Therefore, by \cref{lem:cal_E_t_k} we obtain the theorem.
\end{proof}

\begin{rem}
  \label{rem:mu-infty}
  As pointed out in \cref{rem:phi},  $ \theta \to {(\pi/2)+} $
  as $ \mu_0 \to \infty $. Hence, \cref{eq:esti-I1,eq:I-1} imply that the $
  C_{\alpha,\mu_0} $ in \cref{eq:y-Y} and the $ C_{\alpha,\mu_0} $ in \cref{eq:y-Y-f}
  will both approach infinity as $ \theta \to {(\pi/2)+} $. Analogously, the $
  C_{\alpha,\mu_0} $ in \cref{eq:y-Y-new} will approach infinity as $ \tau^\alpha/h^2
  \to \infty $.
\end{rem}



\subsection{Convergence of the second discretization}
From the proofs of \cref{thm:y-Y-y0,thm:y-Y-f}, it is easily perceived that the fact (cf. \cref{rem:motivation} and \cref{rem:motivation2})
\[
  (\widehat b(z) - z^{\alpha-3})(0) \neq 0
\]
caused $ (3-\alpha) $-order accuracy of the first discretization. This is the
inspiration for the second discretization. Let $ \widehat\beta(z) $ be the discrete
Laplace transform of $ (\beta_k)_{k=0}^\infty $. The definition of the sequence $
(\beta_k)_{k=0}^\infty $ implies
\begin{align*}
  \widehat\beta(z) &= \widehat b(z) + 2\sin(\alpha\pi/2)
  \sum_{k=1}^\infty (2k\pi)^{\alpha-3} \\
  &= \widehat b(z) - \big( \widehat b(z) - z^{\alpha-3} \big)(0)
  \quad\text{(by \cref{eq:wt-b-2}).}
\end{align*}
Hence, $ \big( \widehat\beta(z) - z^{\alpha-3} \big)(0) = 0 $. Finally, by a simple
modification of the proofs of \cref{thm:y-Y-y0,thm:y-Y-f}, we readily obtain the
following error estimate.

\begin{thm}
  \label{thm:y-Y-new}
  For $ k \in \mathbb N_{>0} $,
  \begin{small}
  \begin{equation}
    \label{eq:y-Y-new}
    \snm{y(t_k)\!-\!\mathcal Y_k} \!\leqslant\! C_{\alpha,\mu_0}
    \!\tau^2 \bigg(
      t_k^{-2} \snm{y_0} + t_k^{-1}\! \snm{y_1} +
      t_k^{\alpha-2}\!\snm{f(0)} +
      \int_{0}^{t_k} \! (t_{k+1}-t)^{\alpha-2}
      \snm{f'(t)} \mathrm{d}t
    \bigg).
  \end{equation}
  \end{small}
\end{thm}


\section{Two full discretizations}
\label{sec:main}
Let $ \mathcal K_h $ be a quasi-uniform and shape-regular triangulation of $\Omega $
consisting of $ d $-simplexes, and we use $ h $ to denote the maximum diameter of the
elements in $ \mathcal K_h $. Define
\begin{align*}
  S_h := {}&\left\{
    v_h \in H_0^1(\Omega):\
    v_h|_K \in P_1(K) \quad\forall \,K \in \mathcal K_h
  \right\},
\end{align*}
where $ P_1(K) $ is the set of all linear functions defined on $ K $. Let $ \Delta_h
: S_h \to S_h $ be the usual discrete Laplace operator, namely,
\[
  \dual{-\Delta_h v_h, w_h}_\Omega =
  \dual{\nabla v_h, \nabla w_h}_\Omega
\]
for all $ v_h, w_h \in S_h $. In addition, let $ P_h $ be the $ L^2 $-orthogonal
projection onto $ S_h $.


Assume that $ u_0, u_1 \in L^2(\Omega) $ and
\[
  f \in L^1(0,\infty;L^2(\Omega)) \cap
  {}_0H^{(1-\alpha)/2}(0,\infty;H^{-1}(\Omega)).
\]
Using \cref{discr:ode-1,discr:ode-2} in time and using $ -\Delta_h $ as the
discretization of $ -\Delta $, we obtain two full discretizations of problem
\cref{eq:model} as follows.

\begin{discr}
  \label{discr:1}
  Let $ U_0 = P_hu_0 $; for each $ k \in \mathbb N $, the value of $ U_{k+1} $ is
  determined by
  \begin{small}
  \begin{equation}
    \label{eq:U}
    \begin{aligned}
      & (b_{k+1} - b_k)(U_1-U_0) +
      \sum_{j=1}^k (b_{k-j+1} - b_{k-j}) (U_{j+1} - 2U_j + U_{j-1}) \\
      & \quad {} - \frac{\tau^\alpha}2 \Delta_h (U_k + U_{k+1}) =
      \tau^{\alpha-1} P_h \int_{t_k}^{t_{k+1}} f(t) \, \mathrm{d}t +
      \tau (b_{k+1} - b_k) P_h u_1.
    \end{aligned}
  \end{equation}
  \end{small}
\end{discr}
\begin{discr}
  \label{discr:2}
  Let $ \mathcal U_0 = P_h u_0 $; for each $ k \in \mathbb N $, the value of $
  \mathcal U_{k+1} $ is determined by
  \begin{small}
  \begin{equation}
    \label{eq:new-U}
    \begin{aligned}
      & (\beta_{k+1} - \beta_k) (\mathcal U_1 - \mathcal U_0) +
      \sum_{j=1}^k (\beta_{k-j+1} - \beta_{k-j})
      (\mathcal U_{j+1} - 2\mathcal U_j + \mathcal U_{j-1}) \\
      & \quad {} - \frac{\tau^\alpha}2 \Delta_h
      (\mathcal U_k + \mathcal U_{k+1}) =
      \tau^{\alpha-1} P_h\int_{t_k}^{t_{k+1}} f(t) \, \mathrm{d}t +
      \tau (\beta_{k+1} - \beta_k) P_hu_1.
    \end{aligned}
  \end{equation}
  \end{small}
\end{discr}
\begin{rem}
  We note that \cref{discr:1} has already been analyzed in
  \cite{Luo2018Convergence}, and the following error estimate has been established in
  the case $ u_0 = u_1 = 0 $:
  \begin{small}
  \begin{align*}
    \nm{u(t_k)-U_k}_{H_0^1(\Omega)} \lesssim
    \nm{f}_{L^2(0,t_k;L^2(\Omega))}
    \begin{cases}
      \tau^{(\alpha-1)/2} + h^{1-1/\alpha} &
      \text{ if } 1 < \alpha \leqslant 3/2, \\
      \tau^{(\alpha-1)/2} + \tau^{-1/2} h &
      \text{ if } 3/2 < \alpha < 2,
    \end{cases}
  \end{align*}
  \end{small}
where $ h \leqslant \tau^{\alpha/2} $ if $ 3/2 < \alpha < 2 $.
  This error estimate is optimal with respect to the regularity of $ u $.
\end{rem}

By \cref{lem:Y_k-stability,lem:Y_k-stability-new}, we easily obtain the following
stability estimates of \cref{discr:1,discr:2}.
\begin{thm}
  \label{thm:U-stability}
  For each $ k \in \mathbb N_{>0} $,
  \begin{small}
  \begin{align*}
    \nm{U_k}_{L^2(\Omega)} &\leqslant C_\alpha \big(
      \nm{u_0}_{L^2(\Omega)} +
      t_k^{1-\alpha/2} \nm{u_1}_{\dot H^{-1}(\Omega)} +
      \nm{f}_{{}_0H^{(1-\alpha)/2}(0,t_k;H^{-1}(\Omega))}
    \big), \\
    \nm{\mathcal U_k}_{L^2(\Omega)} &\leqslant C_\alpha \big(
      \nm{u_0}_{L^2(\Omega)} +
      t_k^{1-\alpha/2} \nm{u_1}_{\dot H^{-1}(\Omega)} +
      \nm{f}_{{}_0H^{(1-\alpha)/2}(0,t_k;H^{-1}(\Omega))}
    \big).
  \end{align*}
  \end{small}
\end{thm}
\begin{rem}
 Since we do not use Laplace transform technique in the proof of \cref{lem:Y_k-stability},  the first stability estimate in the above theorem does not require the temporal
  grid to be uniform. We also note that the stability estimate in \cite[Theorem
  3.2]{Sun2006} essentially requires the initial value to be continuously
  differentiable.
\end{rem}

The main task of the rest of this section is to establish the convergence of
\cref{discr:1,discr:2}. To this end, we first introduce the following conventions: $
a \lesssim b $ means that there exists a positive constant $ C $ depending only on $
\alpha $, $ \Omega $, the shape regularity of $ \mathcal K_h $ or $ h_\text{min}^{-2}
\tau^\alpha, $ such that $ a \leqslant C b $, where $ h_\text{min} $ is the minimum
diameter of the elements in $ \mathcal K_h $. Then let us consider the error estimate
of the following spatial semidiscretization of problem \cref{eq:model}:
\begin{equation}
  \label{eq:semidiscr}
  \D_{0+}^{\alpha-1}(u_h' - P_hu_1)(t) -
  \Delta_h u_h(t) = P_h f(t), \quad t > 0,
\end{equation}
subjected to the initial value condition $ u_h(0) = P_hu_0 $.
\begin{lem}
  \label{lem:space}
  If $ u_0, u_1 \in L^2(\Omega) $ and $ f \in L^\infty(0,\infty,L^2(\Omega)) $, then
  \begin{equation}
    \begin{aligned}
      \nm{(u-u_h)(t)}_{L^2(\Omega)} & \lesssim
      h^2 \Big(
        t^{-\alpha} \nm{u_0}_{L^2(\Omega)} +
        t^{1-\alpha} \nm{u_1}_{L^2(\Omega)} \\
        & \qquad\qquad\qquad {} + (1+\snm{\ln h})
        \nm{f}_{L^\infty(0,t;L^2(\Omega))}
      \Big)
    \end{aligned}
  \end{equation}
  for each $ t > 0 $.
\end{lem}
\begin{proof}
 For $f=0$, \cite[Theorem 3.2]{Jin2016} implies
  \[
    \nm{(u-u_h)(t)}_{L^2(\Omega)} \lesssim
    h^2 \big(
      t^{-\alpha} \nm{u_0}_{L^2(\Omega)} +
      t^{1-\alpha} \nm{u_1}_{L^2(\Omega)}
    \big),
  \]
  it suffices to prove, for $u_0=u_1=0$, that
  \begin{equation}
    \label{eq:spatial_f}
    \nm{(u-u_h)(t)}_{L^2(\Omega)} \lesssim
    (1+\snm{\ln h}) h^2 \nm{f}_{L^\infty(0,t;L^2(\Omega))},
  \end{equation}
  which is an improvement of \cite[Theorem 3.3]{Jin2016}. To this end, we proceed as
  follows. Similar to \cite[Equation (25)]{Lubich1996}, we have
  \begin{align*}
    (u-u_h)(t) = \int_0^t \frac1{2\pi i}
    \int_\Upsilon e^{sz} \big(
      (z^\alpha-\Delta)^{-1} - (z^\alpha-\Delta_h)^{-1} P_h)
    \big) \, \mathrm{d}z f(t-s) \, \mathrm{d}s,
  \end{align*}
  where $ \Upsilon $ is defined in \cref{sec:ode}. The proof of \cite[Theorem
  2.1]{Lubich1996} proves that
  \[
    \nm{
      (z^\alpha-\Delta)^{-1} - (z^\alpha-\Delta_h)^{-1}P_h
    }_{\mathcal L(L^2(\Omega))} \lesssim h^2,
    \quad \forall z \in \Upsilon \setminus \{0\},
  \]
  and hence
  \begin{small}
  \begin{align*}
    \Nm{
      \int_\Upsilon e^{sz} \big(
        (z^\alpha-\Delta)^{-1} - (z^\alpha-\Delta_h)^{-1} P_h
      \big) \, \mathrm{d}z
    }_{\mathcal L(L^2(\Omega))} \lesssim s^{-1} h^2.
  \end{align*}
  \end{small}
  We also have
  \begin{small}
  \begin{align*}
    \Nm{
      \int_\Upsilon e^{sz} \big(
        (z^\alpha-\Delta)^{-1} - (z^\alpha-\Delta_h)^{-1} P_h
      \big) \, \mathrm{d}z
    }_{\mathcal L(L^2(\Omega))} \lesssim 1,
  \end{align*}
  \end{small}
  by the fact that, for $ z \in \Upsilon \setminus \{0\} $,
  \begin{align*}
    \nm{ (z^\alpha-\Delta)^{-1} }_{\mathcal L(L^2(\Omega))}
    & \lesssim (1+\snm{z}^\alpha)^{-1}, \\
    \nm{ (z^\alpha-\Delta_h)^{-1} }_{\mathcal L(L^2(\Omega))}
    & \lesssim (1+\snm{z}^\alpha)^{-1}.
  \end{align*}
  Therefore, if $ h^2 < t $ then
  \begin{align*}
    \nm{(u-u_h)(t)}_{L^2(\Omega)} & \lesssim
    \int_0^{h^2} \nm{f(t-s)}_{L^2(\Omega)} \, \mathrm{d}s +
    \int_{h^2}^t s^{-1} h^2  \nm{f(t-s)}_{L^2(\Omega)} \, \mathrm{d}s \\
    & \lesssim h^2(1 + \snm{\ln h}) \nm{f}_{L^\infty(0,t;L^2(\Omega))},
  \end{align*}
  and if $ t \leqslant h^2 $ then
  \begin{align*}
    \nm{(u-u_h)(t)}_{L^2(\Omega)}  \lesssim
    \int_0^t \nm{f(t-s)}_{L^2(\Omega)} \, \mathrm{d}s
    \lesssim h^2\nm{f}_{L^\infty(0,t;L^2(\Omega))}.
  \end{align*}
  This proves \cref{eq:spatial_f} and thus concludes the proof.
\end{proof}
\begin{rem}
  Since
  \begin{small}
  \begin{align*}
    u' - \Delta \D_{0+}^{1-\alpha} u = \D_{0+}^{1-\alpha} f, \\
    u_h' - \Delta_h \D_{0+}^{1-\alpha} u = \D_{0+}^{1-\alpha} P_hf,
  \end{align*}
  \end{small}
  we have
  \begin{align*}
    \dual{u'(t) - u_h'(t), v_h}_\Omega +
    \dual{\nabla \D_{0+}^{1-\alpha}(u-u_h)(t), \nabla v_h}_\Omega = 0
  \end{align*}
  for all $ v_h \in S_h $. Then, by the techniques used in \cref{lem:Y_k-stability},
  a standard energy argument yields
  \begin{small}
  \[
    \nm{(u-u_h)(t)}_{L^2(\Omega)} \leqslant
    2 \nm{(I-R_h)u'}_{L^1(0,t;L^2(\Omega))} +
    \nm{u(t) - R_hu(t)}_{L^2(\Omega)},\quad
    t > 0,
  \]
  \end{small}
  where $ R_h: H_0^1(\Omega) \to S_h $ is defined by that, for each $ v \in
  H_0^1(\Omega) $,
  \[
    \int_\Omega \nabla(v-R_hv) \cdot \nabla w_h = 0
    \quad\text{ for all } w_h \in S_h.
  \]
  We can also use this estimate to analyze the convergence of \cref{eq:semidiscr} in
  $ L^2(\Omega) $-norm with nonsmooth data.
\end{rem}

Finally, let us give the error estimates of \cref{discr:1}. By triangle inequality, we have
	\[ \nm{(u-U)(t)}_{L^2(\Omega)} \leqslant \nm{(u-u_h)(t)}_{L^2(\Omega)}+\nm{(u_h-U)(t)}_{L^2(\Omega)}, \quad   \text{for} \ 0<t \leqslant T,\]
	where $U:=\sum_{k=0}^{\infty} U_k \varphi_k$ and $\varphi_k$ is the hat function at node $t_k$.
	The estimate of $\nm{(u-u_h)(t)}_{L^2(\Omega)}$  already exists (cf. \cref{lem:space}), and hence we only need to give the estimate of $\nm{(u_h-U)(t)}_{L^2(\Omega)}$.
	For $i=1,2,\cdots,N,$ let $(\phi_i,\lambda_i)$ be the eigen-pair of the operator $-\Delta_h$. We have $$u_h=\sum_{i=1}^{N} \dual{u_h,\phi_i}_{\Omega}\phi_i, \quad U=\sum_{i=1}^{N} \dual{U,\phi_i}_{\Omega}\phi_i.$$
	It is easy to verify that $u_h^i:=\dual{u_h,\phi_i}_\Omega$  with $u_h^i(0)=\dual{u_0,\phi_i}_{\Omega}$ satisfies that   \[
	\D^{\alpha-1}_{0+} ((u_h^i)'-u_1^i) + \lambda_i u_h^i =f_i, \quad \text{for} \ i=1,2, \cdots, N,
	\]
	where $f_i=\dual{f,\phi_i}_{\Omega}$ and $u_1^i=\dual{u_1,\phi_i}_{\Omega}$. Letting $U^i:=\dual{U,\phi_i}_\Omega$, by  \cref{thm:y-Y-y0,thm:y-Y-f} we can obtain the error estimates between $u_h^i$ and $U^i$, and hence the error estimates between $u_h$ and $U$. By \cref{thm:y-Y-new}, the error estimates of  \cref{discr:2} follows similarly. By the above procedure, 
	we have the following two theorems.
	%

\begin{thm}
  \label{thm:conv}
  For $ k \in \mathbb N_{>0} $, if $ f' \in L^1(0,t_k;L^2(\Omega)) $ then
  \begin{small}
  \begin{equation}
    \label{eq:conv}
    \begin{aligned}
      & \nm{u(t_k) - U_k}_{L^2(\Omega)} \\
      \lesssim{} &
      \big(
        t_k^{\alpha-3} \tau^{3-\alpha} +
        t_k^{-\alpha} h^2
      \big) \nm{u_{0}}_{L^2(\Omega)} +
      \big(
        t_k^{\alpha-2} \tau^{3-\alpha} +
        t_k^{1-\alpha} h^2
      \big) \nm{u_1}_{L^2(\Omega)} \\
      & {} +
      \tau^{3-\alpha} \varepsilon(\alpha,\tau,k) \nm{f(0)}_{L^2(\Omega)} +
      (1+\snm{\ln h}) h^2 \nm{f}_{L^\infty(0,t_k;L^2(\Omega))} \\
      & {} +
      \tau^{3-\alpha}
      \begin{cases}
        \int_0^{t_k} (t_{k+1}-t)^{2\alpha-3}
        \nm{f'(t)}_{L^2(\Omega)} \, \mathrm{d}t
        & \text{ if } 1 < \alpha < 3/2, \\
        (1+\snm{\ln\tau}) \nm{f'}_{L^1(0,t_k;L^2(\Omega))}
        & \text{ if } \alpha = 3/2, \\
        \nm{f'}_{L^1(0,t_k;L^2(\Omega))}
        & \text{ if } 3/2 < \alpha < 2,
      \end{cases}
    \end{aligned}
  \end{equation}
  \end{small}
  where $ \varepsilon(\alpha,\tau,k) $ is defined by \cref{eq:varepsilon}.
\end{thm}

\begin{thm}
  \label{thm:conv-new}
  For $ k \in \mathbb N_{>0} $, if $ f' \in L^1(0,t_k;L^2(\Omega)) $ then
  \begin{equation}
    \label{eq:conv-new}
    \begin{aligned}
      \nm{u(t_k) - \mathcal U_k}_{L^2(\Omega)}
      & \lesssim \big(\tau^2 t_k^{-2} + t_k^{-\alpha} h^2\big)
      \nm{u_0}_{L^2(\Omega)} + \big(
        \tau^2 t_k^{-1} + t_k^{1-\alpha} h^2
      \big) \nm{u_1}_{L^2(\Omega)} \\
      & \quad{} + \tau^2 t_k^{\alpha-2} \nm{f(0)}_{L^2(\Omega)} +
      (1+\snm{\ln h})h^2 \nm{f}_{L^\infty(0,t_k;L^2(\Omega))} \\
      & \quad{} + \tau^2 \int_0^{t_k}
      (t_{k+1}-t)^{\alpha-2} \nm{f'(t)}_{L^2(\Omega)} \, \mathrm{d}t.
    \end{aligned}
  \end{equation}
\end{thm}

\begin{rem}
  \label{rem:ratio}
  From \cref{rem:mu-infty} it follows that the implicit constants in
  \cref{eq:conv,eq:conv-new} will approach infinity as $ \tau^\alpha/h^2 \to \infty
  $.
\end{rem}

\section{Numerical experiments}
\label{sec:numer}
\subsection{\texorpdfstring{\cref{discr:ode-1,discr:ode-2}}{}}
For equation \cref{eq:ode-y}, we set $ \lambda = 1 $ and consider the following three
problems:
\begin{enumerate}
  \item [(a).] $ y_0 := 1 $, $ y_1 := 0 $, and $ f(t) := 0 $;
  \item [(b).] $ y_0 := 0 $, $ y_1 := 1 $, and $ f(t) := 0 $;
  \item [(c).] $ y_0 := 0 $, $ y_1 := 0 $, and $ f(t) := 1+t^{0.2} $.
\end{enumerate}
In this subsection, ``Error" means the error of the numerical solution at $ t=1 $,
where the reference solution is the numerical solution of \cref{discr:ode-2} with $
\tau=2^{-18} $. The numerical results in \cref{tab:y0,tab:y1,tab:ode-f} demonstrate
that the accuracies of \cref{discr:ode-1,discr:ode-2} are close to $ \mathcal
O(\tau^{3-\alpha}) $ and $ \mathcal O(\tau^2) $, respectively, which agrees well with
\cref{thm:y-Y-y0,thm:y-Y-f,thm:y-Y-new}.

\begin{table}[H]
  \footnotesize \setlength{\tabcolsep}{1.0pt}
  \caption{Convergence history of \cref{discr:ode-1,discr:ode-2} for problem $\mathrm{(a)}$}
  \label{tab:y0}
  \begin{tabular}{ccccccccccccc}
    \toprule
    \multicolumn{7}{c}{\cref{discr:ode-1}} &
    \multicolumn{6}{c}{\cref{discr:ode-2}} \\
    \cmidrule(r){2-7} \cmidrule(r){8-13} &
    \multicolumn{2}{c}{$\alpha=1.2$} &
    \multicolumn{2}{c}{$\alpha=1.4$} &
    \multicolumn{2}{c}{$\alpha=1.8$} &
    \multicolumn{2}{c}{$\alpha=1.2$} &
    \multicolumn{2}{c}{$\alpha=1.4$} &
    \multicolumn{2}{c}{$\alpha=1.8$} \\
    \cmidrule(r){2-3}
    \cmidrule(r){4-5}
    \cmidrule(r){6-7}
    \cmidrule(r){8-9}
    \cmidrule(r){10-11}
    \cmidrule(r){12-13}
    $\tau$      & Error   & Order & Error   & Order & Error   & Order & Error    & Order & Error    & Order & Error    & Order \\
    $2^{-\!10}$ & 2.05e-7 & --    & 8.40e-7 & --    & 4.97e-5 & --    & 6.15e-8  & --    & 4.19e-08 & --    & 8.47e-8  & --    \\
    $2^{-\!11}$ & 6.12e-8 & 1.75  & 2.81e-7 & 1.58  & 2.16e-5 & 1.20  & 1.54e-8  & 2.00  & 1.07e-08 & 1.98  & 1.97e-8  & 2.10  \\
    $2^{-\!12}$ & 1.81e-8 & 1.75  & 9.35e-8 & 1.59  & 9.42e-6 & 1.20  & 3.84e-9  & 2.00  & 2.70e-09 & 1.98  & 4.62e-9  & 2.09  \\
    $2^{-\!13}$ & 5.35e-9 & 1.76  & 3.11e-8 & 1.59  & 4.10e-6 & 1.20  & 9.61e-10 & 2.00  & 6.79e-10 & 1.99  & 1.09e-9  & 2.09  \\
    $2^{-\!14}$ & 1.57e-9 & 1.77  & 1.03e-8 & 1.59  & 1.78e-6 & 1.20  & 2.42e-10 & 1.99  & 1.71e-10 & 1.99  & 2.57e-10 & 2.08  \\
    \bottomrule
  \end{tabular}
\end{table}

\begin{table}[H]
  \footnotesize \setlength{\tabcolsep}{1.0pt}
  \caption{Convergence history of \cref{discr:ode-1,discr:ode-2} for problem $\mathrm{(b)}$}
  \label{tab:y1}
  \begin{tabular}{ccccccccccccc}
    \toprule
    \multicolumn{7}{c}{\cref{discr:ode-1}} &
    \multicolumn{6}{c}{\cref{discr:ode-2}} \\
    \cmidrule(r){2-7} \cmidrule(r){8-13} &
    \multicolumn{2}{c}{$\alpha=1.2$} &
    \multicolumn{2}{c}{$\alpha=1.5$} &
    \multicolumn{2}{c}{$\alpha=1.9$} &
    \multicolumn{2}{c}{$\alpha=1.2$} &
    \multicolumn{2}{c}{$\alpha=1.5$} &
    \multicolumn{2}{c}{$\alpha=1.9$} \\
    \cmidrule(r){2-3}
    \cmidrule(r){4-5}
    \cmidrule(r){6-7}
    \cmidrule(r){8-9}
    \cmidrule(r){10-11}
    \cmidrule(r){12-13}
    $\tau$      & Error   & Order & Error   & Order & Error   & Order & Error    & Order & Error    & Order & Error    & Order \\
    $2^{-\!7}$  & 2.22e-6 & --   & 7.90e-5 & --   & 9.98e-4 & --   & 2.65e-6 & --   & 3.67e-6 & --   & 4.58e-6 & --   \\
    $2^{-\!8}$  & 7.32e-7 & 1.60 & 2.83e-5 & 1.48 & 4.67e-4 & 1.10 & 6.53e-7 & 2.02 & 9.13e-7 & 2.01 & 9.95e-7 & 2.20 \\
    $2^{-\!9}$  & 2.34e-7 & 1.65 & 1.01e-5 & 1.49 & 2.18e-4 & 1.10 & 1.62e-7 & 2.01 & 2.27e-7 & 2.01 & 2.14e-7 & 2.22 \\
    $2^{-\!10}$ & 7.31e-8 & 1.68 & 3.60e-6 & 1.49 & 1.02e-4 & 1.10 & 4.03e-8 & 2.01 & 5.66e-8 & 2.00 & 4.54e-8 & 2.24 \\
    $2^{-\!11}$ & 2.25e-8 & 1.70 & 1.28e-6 & 1.49 & 4.75e-5 & 1.10 & 1.01e-8 & 2.00 & 1.41e-8 & 2.00 & 9.52e-9 & 2.26 \\
    \bottomrule
  \end{tabular}
\end{table}

\begin{table}[H]
  \footnotesize \setlength{\tabcolsep}{1.0pt}
  \caption{Convergence history of \cref{discr:ode-1,discr:ode-2} for problem $\mathrm{(c)}$}
  \label{tab:ode-f}
  \begin{tabular}{ccccccccccccc}
    \toprule
    \multicolumn{7}{c}{\cref{discr:ode-1}} &
    \multicolumn{6}{c}{\cref{discr:ode-2}} \\
    \cmidrule(r){2-7} \cmidrule(r){8-13} &
    \multicolumn{2}{c}{$\alpha=1.2$} &
    \multicolumn{2}{c}{$\alpha=1.4$} &
    \multicolumn{2}{c}{$\alpha=1.9$} &
    \multicolumn{2}{c}{$\alpha=1.2$} &
    \multicolumn{2}{c}{$\alpha=1.4$} &
    \multicolumn{2}{c}{$\alpha=1.9$} \\
    \cmidrule(r){2-3}
    \cmidrule(r){4-5}
    \cmidrule(r){6-7}
    \cmidrule(r){8-9}
    \cmidrule(r){10-11}
    \cmidrule(r){12-13}
    $\tau$      & Error   & Order & Error   & Order & Error   & Order & Error   & Order & Error   & Order & Error   & Order \\
    $2^{-\!7}$  & 1.36e-5 & --    & 2.73e-5 & --    & 2.61e-3 & --    & 5.31e-6 & --    & 2.35e-6 & --    & 9.39e-6 & --    \\
    $2^{-\!8}$  & 4.11e-6 & 1.72  & 9.31e-6 & 1.55  & 1.22e-3 & 1.10  & 1.32e-6 & 2.01  & 6.62e-7 & 1.83  & 2.32e-6 & 2.02  \\
    $2^{-\!9}$  & 1.23e-6 & 1.74  & 3.14e-6 & 1.57  & 5.70e-4 & 1.10  & 3.27e-7 & 2.01  & 1.77e-7 & 1.90  & 5.71e-7 & 2.02  \\
    $2^{-\!10}$ & 3.67e-7 & 1.75  & 1.05e-6 & 1.58  & 2.66e-4 & 1.10  & 8.10e-8 & 2.01  & 4.61e-8 & 1.94  & 1.40e-7 & 2.03  \\
    $2^{-\!11}$ & 1.08e-7 & 1.76  & 3.52e-7 & 1.58  & 1.24e-4 & 1.10  & 2.01e-8 & 2.01  & 1.18e-8 & 1.96  & 3.44e-8 & 2.03  \\
    \bottomrule
  \end{tabular}
\end{table}

\subsection{\texorpdfstring{\cref{discr:1,discr:2}}{}}
For equation \cref{eq:model} in the case $ \Omega = (0,1) $, we consider the
following three problems:
\begin{enumerate}
  \item [(d).] $ u_0(x) := x^{-0.49} $, $ u_1(x) := 0 $, and $ f(x,t) := 0 $;
  \item [(e).] $ u_0(x) := 0 $, $ u_1(x) := x^{-0.49} $, and $ f(x,t) := 0 $;
  \item [(f).] $ u_0(x) := 0 $, $ u_1(x) := 0 $, and $ f(x,t) := x^{-0.49}(1+t^{0.2}) $.
\end{enumerate}
Throughout this subsection, we will use uniform spatial grids, and ``Error1" and
``Error2" denote the errors (in $ L^2(\Omega) $-norm) of the numerical solutions of
\cref{discr:1,discr:2} at $ t=1 $, respectively, where the reference solution is the
numerical solution of \cref{discr:2} with $ h=2^{-11} $ and $ \tau=2^{-16} $.

\medskip\noindent\textbf{Experiment 1.} This experiment verifies the spatial
accuracies of \cref{discr:1,discr:2}. \cref{tab:space} demonstrates that the spatial
accuracy of \cref{discr:1} is close to $ \mathcal O(h^2) $, which is in good
agreement with \cref{thm:conv}. Since the numerical results of \cref{discr:2} are
almost identical to that of \cref{discr:1}, they are omitted here.

\begin{table}[H]
  \caption{Convergence history of \cref{discr:1} for problems $\mathrm{(d)}$,
  $\mathrm{(e)}$ and $ \mathrm{(f)} $ with $ \tau = 2^{-16} $}
  \label{tab:space}
  \footnotesize
  \setlength{\tabcolsep}{6pt}
  \begin{tabular}{cccccccc}
    \toprule & &
    \multicolumn{2}{c}{$\alpha=1.2$} &
    \multicolumn{2}{c}{$\alpha=1.4$} &
    \multicolumn{2}{c}{$\alpha=1.8$} \\
    \cmidrule(r){3-4}
    \cmidrule(r){5-6}
    \cmidrule(r){7-8}
    & $h$ & Error1 & Order & Error1 & Order & Error1 & Order \\
    \midrule
    \multirow{5}{*}{Problem (d)}
    & $2^{-3}$ & 1.07e-3 & --   & 4.43e-3 & --   & 4.74e-2 & --   \\
    & $2^{-4}$ & 2.73e-4 & 1.97 & 1.12e-3 & 1.99 & 1.57e-2 & 1.59 \\
    & $2^{-5}$ & 6.94e-5 & 1.98 & 2.80e-4 & 2.00 & 4.46e-3 & 1.82 \\
    & $2^{-6}$ & 1.76e-5 & 1.98 & 7.00e-5 & 2.00 & 1.15e-3 & 1.96 \\
    & $2^{-7}$ & 4.45e-6 & 1.98 & 1.75e-5 & 2.00 & 2.85e-4 & 2.01 \\
    \midrule
    \multirow{5}{*}{Problem (e)}
    & $2^{-3}$ & 2.71e-3 & --   & 2.33e-3 & --   & 7.76e-3 & --   \\
    & $2^{-4}$ & 7.21e-4 & 1.91 & 6.15e-4 & 1.92 & 2.04e-3 & 1.93 \\
    & $2^{-5}$ & 1.90e-4 & 1.92 & 1.61e-4 & 1.93 & 5.10e-4 & 2.00 \\
    & $2^{-6}$ & 4.97e-5 & 1.93 & 4.19e-5 & 1.94 & 1.27e-4 & 2.00 \\
    & $2^{-7}$ & 1.29e-5 & 1.94 & 1.08e-5 & 1.95 & 3.16e-5 & 2.00 \\
    \midrule
    \multirow{5}{*}{Problem (f)}
    & $2^{-3}$ & 6.12e-3 & --   & 6.64e-3 & --   & 8.77e-3 & --   \\
    & $2^{-4}$ & 1.63e-3 & 1.91 & 1.75e-3 & 1.92 & 2.27e-3 & 1.95 \\
    & $2^{-5}$ & 4.31e-4 & 1.92 & 4.60e-4 & 1.93 & 5.86e-4 & 1.96 \\
    & $2^{-6}$ & 1.13e-4 & 1.93 & 1.20e-4 & 1.94 & 1.51e-4 & 1.96 \\
    & $2^{-7}$ & 2.95e-5 & 1.94 & 3.12e-5 & 1.95 & 3.89e-5 & 1.95 \\
    \bottomrule
  \end{tabular}
\end{table}



\noindent\textbf{Experiment 2.} To obtain the temporal accuracies $ \mathcal
O(\tau^{3-\alpha}) $ and $ \mathcal O(\tau^2) $ of \cref{discr:1,discr:2},
respectively, \cref{thm:conv,thm:conv-new} require the ratio $ \tau^\alpha/h^2 $ to
be uniformly bounded. Hence, this experiment verifies the temporal accuracies of
\cref{discr:1,discr:2} in an indirect way. In this experiment, we set $ \tau^\alpha =
h^2 $. \cref{thm:conv} predicts that ``Error1" is close to $ \mathcal O(h^2) $ for $
1 < \alpha \leqslant 3/2 $ and close to $ \mathcal O(h^{6/\alpha-2}) $ for $ 3/2 <
\alpha < 2 $. \cref{thm:conv-new} predicts that ``Error2" is close to $ \mathcal
O(h^2) $ for all $ 1 < \alpha < 2 $. The above two predictions are confirmed by the
numerical results in \cref{tab:ex2-d,tab:ex2-e,tab:ex2-f}.

\begin{table}[H]
  \footnotesize \setlength{\tabcolsep}{1.0pt}
  \caption{Convergence history of \cref{discr:1,discr:2} for problem $\mathrm{(d)}$}
  \label{tab:ex2-d}
  \begin{tabular}{ccccccccccccc}
    \toprule &
    \multicolumn{4}{c}{$\alpha=1.2$} &
    \multicolumn{4}{c}{$\alpha=1.5$} &
    \multicolumn{4}{c}{$\alpha=1.8$} \\
    \cmidrule(r){2-5}
    \cmidrule(r){6-9}
    \cmidrule(r){10-13}
    $h$      & Error1   & Order & Error2   & Order & Error1   & Order & Error2    & Order & Error1    & Order & Error2    & Order \\
    $2^{-5}$ & 6.87e-5 & --   & 6.99e-5 & --   & 4.83e-4 & --   & 5.00e-4 & --   & 4.19e-2 & --   & 7.12e-3 & --   \\
    $2^{-6}$ & 1.75e-5 & 1.97 & 1.76e-5 & 1.99 & 1.25e-4 & 1.95 & 1.25e-4 & 2.00 & 1.98e-2 & 1.08 & 1.27e-3 & 2.48 \\
    $2^{-7}$ & 4.44e-6 & 1.98 & 4.45e-6 & 1.99 & 3.16e-5 & 1.98 & 3.10e-5 & 2.01 & 8.57e-3 & 1.21 & 2.59e-4 & 2.29 \\
    $2^{-8}$ & 1.11e-6 & 1.99 & 1.12e-6 & 2.00 & 8.08e-6 & 1.97 & 7.65e-6 & 2.02 & 3.54e-3 & 1.27 & 5.88e-5 & 2.14 \\
    \bottomrule
  \end{tabular}
\end{table}

\begin{table}[H]
  \footnotesize \setlength{\tabcolsep}{1.0pt}
  \caption{Convergence history of \cref{discr:1,discr:2} for problem $\mathrm{(e)}$}
  \label{tab:ex2-e}
  \begin{tabular}{ccccccccccccc}
    \toprule
    & \multicolumn{4}{c}{$\alpha=1.2$} &
    \multicolumn{4}{c}{$\alpha=1.5$} &
    \multicolumn{4}{c}{$\alpha=1.8$} \\
    \cmidrule(r){2-5}
    \cmidrule(r){6-9}
    \cmidrule(r){10-13}
    \cmidrule(r){10-11}
    $h$      & Error1  & Order & Error2  & Order & Error1  & Order & Error2  & Order & Error1  & Order & Error2  & Order \\
    $2^{-5}$ & 1.90e-4 & --   & 1.90e-4 & --   & 1.17e-4 & --   & 1.71e-4 & --   & 6.12e-3 & --   & 6.06e-4 & --   \\
    $2^{-6}$ & 4.98e-5 & 1.93 & 4.97e-5 & 1.93 & 3.14e-5 & 1.90 & 4.23e-5 & 2.01 & 2.60e-3 & 1.24 & 1.08e-4 & 2.48 \\
    $2^{-7}$ & 1.29e-5 & 1.94 & 1.29e-5 & 1.94 & 8.25e-6 & 1.93 & 1.06e-5 & 1.99 & 1.06e-3 & 1.29 & 1.96e-5 & 2.46 \\
    $2^{-8}$ & 3.33e-6 & 1.96 & 3.33e-5 & 1.96 & 2.15e-6 & 1.94 & 2.68e-6 & 1.99 & 4.26e-4 & 1.31 & 4.44e-6 & 2.15 \\
    \bottomrule
  \end{tabular}
\end{table}

\begin{table}[H]
  \footnotesize \setlength{\tabcolsep}{1.0pt}
  \caption{Convergence history of \cref{discr:1,discr:2} for problem $\mathrm{(f)}$}
  \label{tab:ex2-f}
  \begin{tabular}{ccccccccccccc}
    \toprule &
    \multicolumn{4}{c}{$\alpha=1.2$} &
    \multicolumn{4}{c}{$\alpha=1.5$} &
    \multicolumn{4}{c}{$\alpha=1.9$} \\
    \cmidrule(r){2-5}
    \cmidrule(r){6-9}
    \cmidrule(r){10-13}
    $h$      & Error1  & Order & Error2  & Order & Error1  & Order & Error2  & Order & Error1  & Order & Error2  & Order \\
    $2^{-4}$ & 1.63e-3 & --   & 1.63e-3 & --   & 1.53e-3 & --   & 1.87e-3 & --   & 2.22e-2 & --   & 2.61e-3 & --   \\
    $2^{-5}$ & 4.31e-4 & 1.92 & 4.31e-4 & 1.92 & 4.06e-4 & 1.92 & 4.91e-4 & 1.93 & 1.03e-2 & 1.10 & 7.35e-4 & 1.83 \\
    $2^{-6}$ & 1.13e-4 & 1.93 & 1.13e-4 & 1.93 & 1.07e-4 & 1.92 & 1.28e-4 & 1.94 & 4.75e-3 & 1.12 & 1.79e-4 & 2.03 \\
    $2^{-7}$ & 2.95e-5 & 1.94 & 2.95e-5 & 1.94 & 2.80e-5 & 1.93 & 3.30e-5 & 1.95 & 2.15e-3 & 1.15 & 4.17e-5 & 2.11 \\
    \bottomrule
  \end{tabular}
\end{table}

\noindent\textbf{Experiment 3.} This experiment investigates the effect of large
ration $ \tau^\alpha/h^2 $ on the accuracy of \cref{discr:1,discr:2} for problem (d).
The numerical results in \cref{tab:ratio} illustrate that, with fixed $ \tau $, the
accuracy of \cref{discr:1,discr:2} will deteriorate as $ h \to {0+} $, which confirms
\cref{rem:ratio}.

\begin{table}[H]
  \caption{Convergence history of \cref{discr:1,discr:2} for problem $ \mathrm{(d)} $
  with $ \tau = 2^{-5} $}
  \label{tab:ratio}
  \footnotesize
  \setlength{\tabcolsep}{3pt}
  \begin{tabular}{ccccccc}
    \toprule &
    \multicolumn{2}{c}{$\alpha=1.2$} &
    \multicolumn{2}{c}{$\alpha=1.4$} &
    \multicolumn{2}{c}{$\alpha=1.8$} \\
    \cmidrule(r){2-3}
    \cmidrule(r){4-5}
    \cmidrule(r){6-7}
    $h$ & $ \text{Error1} $ & Error2
    & $ \text{Error1} $ & Error2
    & $ \text{Error1} $ & Error2 \\
    $2^{-4}$ & 3.20e-3 & 5.90e-4 & 1.29e-3 & 1.52e-3 & 5.82e-2 & 1.12e-2 \\
    $2^{-5}$ & 1.04e-1 & 6.00e-2 & 1.72e-2 & 4.05e-3 & 6.13e-2 & 2.25e-2 \\
    $2^{-6}$ & 3.81e-1 & 3.04e-1 & 1.87e-1 & 1.17e-1 & 6.29e-2 & 2.60e-2 \\
    $2^{-7}$ & 7.04e-1 & 6.26e-1 & 4.94e-1 & 4.00e-1 & 1.59e-1 & 6.81e-2 \\
    $2^{-8}$ & 9.97e-1 & 9.28e-1 & 8.09e-1 & 7.20e-1 & 4.44e-1 & 3.09e-1 \\
    $2^{-9}$ & 1.25e-0 & 1.19e-0 & 8.09e-1 & 1.01e-0 & 7.58e-1 & 6.27e-1 \\
    \bottomrule
  \end{tabular}
\end{table}

\section{Conclusion}
\label{sec:conclusion}
The well-known L1 scheme for fractional wave equations is analyzed in this paper. New
stability estimate is established, and temporal accuracy $ \mathcal
O(\tau^{3-\alpha}) $ is derived for nonsmooth initial values $ u_0 $ and $ u_1 $. A
modified L1 scheme is also proposed, which possesses temporal accuracy $ \mathcal
O(\tau^2) $. The theoretical results reveal that $ \tau^\alpha / h_{\min}^2 $ should
be uniformly bounded, where $ h_{\min} $ is the minimum diameter of the elements in $
\mathcal K_h $; otherwise, the temporal accuracy will deteriorate. Numerical
experiments are performed to verify the theoretical results.

If the temporal grid is nonuniform or the governing equation is of the form
\[
  \D_{0+}^{\alpha-1}(u'-u_1)(t) - \mathrm{div}(a(x,t) \nabla u(t)) = f(t),
  \quad t > 0,
\]
then the techniques used in this paper can not be applied. Hence, an interesting
question is, on the nonuniform temporal grid or for the above equation, how to derive
sharp error estimates for the L1 scheme with nonsmooth data. This will be our future
work.

\end{document}